\documentclass[12pt]{amsart}

\usepackage{tasks}
\usepackage{hyperref}
\usepackage{blindtext}
\usepackage{setspace}
\usepackage[utf8]{inputenc}
\usepackage{amsmath}
\usepackage{amssymb}

\usepackage{graphicx}
 
\usepackage{subcaption}
\usepackage{placeins} 
\usepackage[utf8]{inputenc}
\usepackage[export]{adjustbox}
\usepackage{wrapfig}

\usepackage{marginnote}
\usepackage{datetime}
\usepackage{color}
\usepackage[nobysame]{amsrefs}
\usepackage{amsmath}
\usepackage{mathtools}
\usepackage{graphicx}  
\usepackage[top=3.5cm, bottom=2.5cm, outer=2.5cm, inner=2.5cm]{geometry}
\usepackage{hyperref}
\hypersetup{linkcolor=red,
citecolor=black
}
\usepackage{amsmath}

\usepackage{empheq}

\newtheorem{thm}{Theorem}[section]

\theoremstyle{definition}

\newtheorem{Alg}[thm]{Algorithm}

\usepackage{hyperref}
\usepackage{hyperref}
\hypersetup{
	colorlinks = true,
	linkcolor = {red},
	citecolor = {blue}
}

\theoremstyle{remark}
\newtheorem{re}[thm]{Remark}

\newcommand{\RNum}[1]{\uppercase\expandafter{\romannumeral #1\relax}}

\numberwithin{equation}{section}

\usepackage[utf8]{inputenc}
\usepackage[english]{babel}
\usepackage{paralist}
\usepackage{amsthm}
\usepackage{marginnote}

\theoremstyle{definition}

\theoremstyle{remark}

 \usepackage{bm,amsmath}
\renewcommand{\vec}[1]{\bm{#1}}

\newcommand{\RN}[1]{%
  \textup{\uppercase\expandafter{\romannumeral#1}}%
}

\newcommand{\Rey}{\mathcal{R}e }

\newcommand{\bff}{\mathbf{f}}
\newcommand{\bfu}{\mathbf{u}}
\newcommand{\bfv}{\mathbf{v}}
\newcommand{\bfw}{\mathbf{w}}
\newcommand{\bfX}{\mathbf{X}}

\newcommand{\bfV}{\mathbf{V}}

\newcommand{\bfx}{\mathbf{x}}

\newcommand{\bfn}{\mathbf{n}}

\def \lb {\langle}
\def \rb {\rangle}
 
\newenvironment{Alirev}{\color{blue}}{\color{black}}
\newcommand{\AAA}{\begin{Alirev}}
\newcommand{\PPP}{\end{Alirev}}

\allowdisplaybreaks[4]

\usepackage[utf8]{inputenc}

\title[DA in LES: Handling Model-Observation Mismatch]{Data Assimilation in Large Eddy Simulation: Addressing Model-Observation Mismatch from Navier-Stokes Data}

\author{Adam Larios}
\address[Adam Larios]{Department of Mathematics, 
                 University of Nebraska--Lincoln,
                 Lincoln, NE 68588-0130, USA}
\email{alarios@unl.edu}

\author{Ali Pakzad}
\address[Ali Pakzad]{Department of Mathematics, 
                California State University Northridge,
        Northridge, CA 91330, USA}
\email{pakzad@csun.edu}

\author{Nicholas White}
\address[Nicholas White]{Department of Mathematics, 
                University of Nebraska--Lincoln,
        Lincoln, NE 68588-0130, USA}
\email{nwhite17@huskers.unl.edu}

\date{\today}

\begin{document}

\begin{abstract}
In atmospheric and turbulent flow modeling, Large Eddy Simulation (LES) is often used to reduce computational cost, while observational data typically originates from the underlying physical system. Motivated by this setting, we study a continuous data assimilation (CDA) algorithm applied to a Smagorinsky/Ladyzhenskaya-type LES model, in which the observational data is generated from the full Navier--Stokes equations (NSE). In the two-dimensional setting, we establish global well-posedness of the assimilated system and prove exponential convergence to the true solution, up to an error of order $\bar{\nu}^{1/2}$, where $\bar{\nu}$ is the turbulence viscosity parameter. In addition to rigorous analysis in 2D, we provide numerical simulations in both 2D domains with physical boundary conditions and 3D periodic domains, demonstrating effective synchronization in these cases, and corroborating our theoretical predictions.
\end{abstract}

\maketitle
\tableofcontents

\section{Introduction}\noindent
Combining incomplete observational data with simulation and dynamical models through data assimilation (DA) techniques is a common practice to improve weather forecasts \cite{K03}. In real-world weather simulations, Large Eddy Simulation (LES) models are often preferred over the full Navier-Stokes equations, as they reduce computational costs and stabilize under-resolved simulations \cite{P04}. Combining data assimilation with LES is a natural extension of this approach (see, e.g., \cites{Albanez_Nussenzveig_Lopes_Titi_2016,Cao_Giorgini_Jolly_Pakzad_2022,Chen_Li_Lunasin_2021,Larios_Pei_2018_NSV_DA}).  However, a challenge arises because observational data does not directly come from the LES model, but rather from the underlying physical model, which one can take to be, e.g., the Navier-Stokes Equations (NSE). With a focus on a particular DA algorithm, this paper investigates the mismatches between observational data (originating from the NSE) and an LES model. We provide an in-depth analysis of the errors that result from this mismatch, including both theoretical analysis and simulation results. 

\subsection{Previous works}
\addcontentsline{toc}{section}{1.1. Previous works}
The concept of using data to enhance the predictive capabilities of mathematical models dates back to Kalman’s groundbreaking work in 1960 \cite{K60}. DA is a critical tool in numerical weather prediction and turbulence modeling that builds on this idea by linking observational data with mathematical models to improve forecast accuracy. DA as a formal methodology emerged in the 1960s when Charney, Halem, and Jastrow \cite{Charney-Use1969} proposed using atmospheric motion equations to incorporate observations of the evolving atmospheric state for more accurate predictions. 
Classical DA methods are generally statistical in nature, such as the Kalman filter  \cite{K60} and its variants \cite{Law_Stuart_Zygalakis_2015_book}, as well as 3DVAR, 4DVAR, and their various adaptations (see, e.g., \cites{K03,Law_Stuart_Zygalakis_2015_book} and the references therein).  However, non-statistical approaches have recently gained in popularity due largely to the advent of interpolated nudging, introduced by Azouani, Olson, and Titi in 2014 \cites{AOT14,Azouani_Titi_2014}, referred to as the ``AOT algorithm'' or ``Continuous Data Assimilation'' (CDA). An earlier form of nudging was originally introduced in \cites{Anthes_1974_JAS,Hoke_Anthes_1976_MWR}, but had many performance issues (see, e.g., \cites{lakshmivarahan2013nudging,brocker2012sensitivity}).  In contrast, the AOT/CDA approach, which was inspired by the work of Foias and Prodi on determining modes \cite{Foias-Sur1967}, is robust, accurate, and has been successfully adapted to many physical settings, as we discuss below.

The Azouani-Olson-Titi (AOT) algorithm (also referred to as continuous data assimilation, or CDA) is  motivated by the practical difficulties in accurately initializing numerical simulations. Observational data often provides an incomplete representation of the true state of the atmosphere, and as a result, the \emph{``correct''} initial conditions for weather prediction are not available \emph{a priori}. This lack of precision in the initial conditions can lead to an exponential growth of errors due to the nonlinear chaotic nature of the underlying system. To address this, the AOT algorithm incorporates observational data directly at the partial differential equation (PDE) level by introducing an interpolated feedback control term, sometimes called ``interpolated nudging.'' This mechanism synchronizes the computed solution with the true solution corresponding to the observed data, ensuring improved predictive accuracy. The algorithm operates on a dissipative dynamical system, which we write abstractly as
\[
\frac{du}{dt} = F(u),
\]  
with an unknown initial condition (note that $F$ can be nonlinear, nonlocal, and can depend also on spatial derivatives of $u$). Here,  we assume that $u(t)$ is our true solution and some information about it is available, such as measurements at specific points (e.g. the locations of the weather stations or satellites) within the domain. From this data, we can construct an approximation of \(u(t)\), denoted by \(I_h(u(t))\), where $h>0$ is a characteristic length scale representing, e.g., the average spacing between observation points, and  $I_h: L^2(\Omega) \rightarrow L^2(\Omega) $ is a linear operator based on roughly  \(\mathcal{O}(h^{-1})\) data points.  We make the following assumptions on $I_h$, which are standard (see, e.g., \cite{AOT14}).  Namely, we assume that there exist constants $c_0>0$ and $c_I>0$, independent of $h$, such that  
\begin{equation}\label{I_h}
\begin{alignedat}{2}
\|I_h \varphi \|_{L^2(\Omega)} 
& 
\leq c_I \|\varphi\|_{L^2(\Omega)},  
& \quad &   
\text{for all}\quad \varphi \in  L^2(\Omega),  
\\ 
 \|\varphi  - I_h\, \varphi\|_{L^2(\Omega)} 
 & 
 \leq c_0 \, h \| \nabla\varphi\|_{L^2(\Omega)},
 & \quad &   
 \text{for all}\quad \varphi \in  H^1(\Omega).
\end{alignedat}
\end{equation}
These observations are integrated into the system through a feedback control term, resulting in the modified system:  
\begin{equation}\label{AOT}
\frac{dv}{dt} = F(v) + \mu \big(I_h(u) - I_h(v)\big), \quad v(t=0) = v_0,
\end{equation}
where \(\mu > 0\) is a positive relaxation (nudging) parameter chosen appropriately, and \(v_0\) is any admissible\footnote{e.g., in the case of the Navier-Stokes equations, admissibilty means that $v_0$ is assumed to satisfy boundary conditions (and a mean-free condition in the case of periodic boundary conditions), be sufficiently smooth, and be divergence-free} initial condition.  For instance, one could choose $v_0\equiv0$.  The goal is to pick $\mu >0$ and $h>0$ such that
$$ \bfv(t)\rightarrow \bfu(t),$$
expoenentially fast as $t \rightarrow \infty$ in a suitable function space.  Examples of such interpolation operators $I_h$ satisfying \eqref{I_h}  include projection onto Fourier modes with wave numbers \(|k| \leq 1/h\), volume elements, and constant finite element interpolation \cite{Biswas-Data2021, AOT14}. Numerical evidence also suggests that higher-order interpolation can accelerate synchronization through nudging \cite{JP23}.  Nonlinear nudging can also yield in super-exponential convergence \cite{Carlson_Larios_Titi_2023_nlDA,Du_Shiue_2021,Larios_Pei_2017_KSE_DA_NL}.

In the context of the 2D incompressible Navier-Stokes equations with both no-slip and periodic boundary conditions, the global well-posedness of \eqref{AOT} and exponential convergence in time to the reference solution were proven in \cite{AOT14}. Numerical experiments have been successfully conducted to test this algorithm on various nonlinear systems, including the 2D Navier-Stokes equations \cite{Gesho-Acomputational2016, LRZ19, Garcia-Archilla-Uniform2020, JP23}, the Rayleigh-Bénard equations \cite{Farhat-Assimilation2018, Altaf-Downscaling2017}, and the Kuramoto-Sivashinsky equations \cite{LT17}. Applications of continuous data assimilation to other physical systems and partial differential equations include non-Newtonian fluids \cite{Cao_Giorgini_Jolly_Pakzad_2022}, magnetohydrodynamic equations \cite{Biswas-Continuous2018}, two-phase models \cite{CLP22},  the Leray-$\alpha$ model of turbulence \cite{Jolly-Adata2017}, the quasi-geostrophic equation \cite{Jolly-Adata2017}, Darcy’s equation \cite{Markowich-Continuous2019}, KdV equations \cite{Jolly-Determining2017}, and the primitive equations \cite{Pei-Continuous2019}, among others. While these studies generally assume noise-free observations, the method has also been extended to handle noisy data, as shown in \cite{Bessaih-Continuous2015, Jolly-Continuous2019, Foias-ADiscrete2016}.  Data Assimilation has also recently found a wide range of applications in  parameter estimation \cite{CHLMNW22, FLMW24, Newey2024}. For further details, readers are referred to other recent works on this topic, such as \cite{Auroux-ANudging2008, Biswas-Data2021, Farhat-Data2020, Farhat-Continuous2015, Farhat-Continuous2017, Farhat-Data2016, FJT, FLT, Foias-ADiscrete2016, Hudson-Numerical2019, Ibdah-Fully2020, Mondaini-Uniform2018, Pawar-Long2020, ZRSI19}, and the references therein.

\subsection{Mathematical Models and PDEs}
\addcontentsline{toc}{section}{1.2. Mathematical Models and PDEs}

The incompressible, constant-density Navier-Stokes equations (NSE) constitute the basic mathematical model for fluid flow, and  contain turbulent dynamics among their solutions \cite{FMRT01}.  They are given by
\begin{equation}
 \label{NSE}
 \begin{split}
\partial_t \bfu+   (\bfu \cdot \nabla) \bfu   - \nabla \cdot  \mathbf{T}(\nabla \bfu) \ +  \nabla p  & = \bff,\\ 
\nabla \cdot \bfu  &= 0,
\end{split}
\qquad \qquad \qquad \text{in } \, \Omega \times (0,\infty),
\end{equation}
with stress tensor given by
 $$\mathbf{T}_{\text{NSE}}(\nabla \bfu)= \nu \nabla \bfu,$$
over a bounded open domain $\Omega\subset\mathbb{R}^d$, $d\in\{2,3\}$,  
where $\bfu$ is the velocity, $\bff = \bff(x, t)$ is the known body force, $p$ is the pressure, and $\nu$ is the kinematic viscosity.   
Here the initial condition $\bfu(\bfx, 0)= \bfu_0(x)$ is not assumed to be known at every point in the domain.  
In the case of physical boundaries, we consider ``no slip'' homogeneous Dirichlet conditions $\bfu\Big|_{\partial\Omega}=\mathbf{0}$.  
In the case of periodic boundary conditions ($\Omega= \mathbb{T}^d= \mathbb{R}^d/\mathbb{Z}^d$), we assume the mean-free condition 
$\int_\Omega\bfu(\bfx,t)\,d\bfx = \mathbf{0}$.

In many practical applications of fluid dynamics, direct numerical simulation (DNS) of the Navier-Stokes equations (NSE) is computationally prohibitive due to the wide range of spatial and temporal scales that must be resolved. The computational cost increases dramatically with the Reynolds number, as finer resolution is required to capture the smallest dissipative scales governed by Kolmogorov theory. Specifically, in fully developed turbulence, the smallest scales, known as the Kolmogorov length scale \(\eta = (\nu^3/\epsilon)^{1/4}\) \cite{F95}, become increasingly small as the Reynolds number grows, necessitating an extremely fine computational grid. The number of degrees of freedom required for a DNS scales as \(\text{Re}^{9/4}\) for three-dimensional turbulence \cite{L08}, making high-Reynolds-number simulations infeasible even with modern supercomputers. For example, simulating isotropic turbulence at a Reynolds number of \(96,000\) via DNS would require approximately 5000 years of computational time \cite[Page 349]{Pope_2000_bible}. This computational intractability underscores the need for alternative modeling approaches.

Large Eddy Simulation (LES) provides a computationally feasible alternative by explicitly resolving only the large-scale motions while modeling the effects of the unresolved subgrid-scale (SGS) dynamics through additional closure terms.    To address this, several modifications to the standard NSE have been proposed, especially in the mid-1960s, by pioneering works of Ladyzhenskaya \cite{Lady67,Lady68} and Smagorinsky \cite{SM63}.  One notable modification is the introduction of a nonlinear form for the Cauchy stress tensor in \eqref{NSE}, which is typically given by:
\begin{equation*}
\label{T-LES}
\mathbf{T}_{\text{LES}} (\nabla \bfu) = \left( \nu  + \bar{\nu} \,  |\nabla \bfu|_F^{p-2} \right)\, \nabla \bfu,
\end{equation*}
where \(\nu\) is the usual  molecular viscosity and \(\bar{\nu}\) is the effective viscosity accounting for the unresolved scales. The term \(|\nabla \bfu|_F^{p-2}\) represents the gradient of the velocity field raised to a power \(p-2\), introducing a nonlinear correction that accounts for the turbulent energy cascades at smaller scales.   LES is widely recognized as one of the most promising approaches \cite{PeleLM-FDF2022, Sharan2019ShearLayer} to significantly reducing computational demand while still capturing the essential turbulent features of the flow \cite{ Wilcox06}. 

To apply the Data Assimilation (DA) algorithm within the Large Eddy Simulation (LES) framework, it is essential to account for the fact that observational data is typically obtained from real-world measurements or high-fidelity simulations based on the full Navier-Stokes Equations. With this in mind, we propose the following nonlinear system over the domain \(\Omega\): 

\begin{equation}
\label{DALady}
 \begin{split}
\partial_t \mathbf{v} + (\mathbf{v} \cdot \nabla) \mathbf{v} - \nu \Delta \mathbf{v} - \nabla \cdot \left( \bar{\nu} \, |\nabla \mathbf{v}|_F^{p-2} \, \nabla \mathbf{v} \right) + \nabla q &= \mathbf{f} + \mu \, I_h(\underbrace{\bfu}_{\text{NSE}}) - \mu \, I_h(\mathbf{v}), \\
\nabla \cdot \mathbf{v} &= 0,
\end{split}
\end{equation}
where the initial condition \(\mathbf{v}_0\) is arbitrary, and \(\mathbf{u}\) represents the solution to the underlying physical model- the Navier-Stokes Equations \eqref{NSE}, from which an interpolation of the observed data is constructed, denoted as \(I_h(\mathbf{u}(t))\).

\subsection{Results of this paper}
\addcontentsline{toc}{section}{1.3. Results of this paper}
 
In this paper, we conduct theoretical analysis and computational  investigations   for the nonlinear nudged  system \eqref{DALady}, examining its synchronization properties with respect to the ``true'' solution of the Navier-Stokes equations. The primary focus is on the two-dimensional (2D) case due to the required regularity of NSE solutions. However, the analysis can be extended to three dimensions (3D), provided that additional regularity assumptions are imposed on the 3D NSE (see \cite{Biswas_Price_2020_AOT3D} for the AOT algorithm in the 3D NSE).

We first provide a formal proof of the global existence of solutions to the system \eqref{DALady} using a priori estimates, which can be made rigorous using, e.g., Galerkin methods.  We then demonstrate that, under appropriate conditions on the nudging parameters \( h \) and \( \mu \), and given a sufficiently developed\footnote{That is, there exists a time $t_0>0$ depending on $\|\bfu_0\|_{L^2}$, such that exponentially fast synchronization happens for $t>t_0$.  In numerical simulations, we observe the exponential rate almost immediately.} reference flow, the nudged solution \( \mathbf{v} \) asymptotically synchronizes exponentially fast in time with the Navier-Stokes solution \( \mathbf{u} \), up to an error that is sublinearly bounded by the turbulence viscosity coefficient \( \bar{\nu} \). Furthermore, as \( \bar{\nu} \to 0 \), this error vanishes as\footnote{We write $a\lesssim b$ to denote $a\leq C b$ for some $C>0$ independent of $a$ and $b$.}
 \begin{equation*}
\|\bfu - \bfv\|_{L^2} \lesssim   \, \, \bar{\nu}^{1/2}\hspace{0.5cm} \text{as}  \hspace{0.5cm} t \rightarrow \infty.
\end{equation*} 
A key innovation of this work is the assimilation of observational data derived from the full NSE rather than from another LES system. This approach introduces significant analytical challenges, particularly due to the presence of the nonlinear \( p \)-Laplacian term in the equation \eqref{DALady}, which does not appear in the reference system \eqref{NSE}. Consequently, the error equation for \( \mathbf{w} = \mathbf{u} - \mathbf{v} \) involves a turbulent viscosity term of the form 
\[
\nabla \cdot ( |\nabla \mathbf{v}|^{p-2} \nabla \mathbf{v} ),
\]
which depends solely on \( \mathbf{v} \) rather than on both \( \mathbf{u} \) and \( \mathbf{v} \). This structural asymmetry prevents the direct application of standard monotonicity arguments typically used in error analysis (see, e.,g \cite{Lady68}). To overcome this difficulty, we reformulate the  above problematic term in terms of \( \mathbf{u} \) and \( \mathbf{w} \), which introduces additional  nonlinear terms that must be carefully controlled. 

We also  assess the performance of our algorithm through computational experiments. Prior numerical investigations on similar systems have shown that the nudging method achieves synchronization with observational data, even without the strict constraints on \( \mu \) and \( h \) required by theoretical estimates \cite{JP23, Farhat-Data2020, Farhat-Assimilation2018, ZRSI19, FLMW24 }. Our results exhibit a comparable behavior in the case of nonhomogeneous boundary conditions, where we observe exponential convergence up to a threshold that matches our analytical prediction, specifically $\|\bfu - \bfv\| \sim \bar{\nu}^{1/2}.$  All simulations are performed for \( p = 3 \), which corresponds to the Smagorinsky model in Large Eddy Simulation  \cite{SM63, L08}, and Germano dynamic model \cite{GPMC91}.  Note that the present work is one of the few works currently includes simulations of an AOT-based algorithm in the setting of physical boundary conditions (see also \cite{JP23, Cibik2025, Cibik2025NA}). While the theoretical results in this work are established in two-dimensional space, owing to the lack of known global regularity for the 3D Navier–Stokes equations, we also complement our analysis with numerical simulations in the three-dimensional periodic domain to explore the behavior of the proposed data assimilation algorithm in that setting.

\section{Notation and Preliminaries}\label{S2}
 Let $\Omega \subset \mathbb{R}^d$, with $d = 2,3$, be a bounded open domain with Lipschitz boundary and volume denoted by $|\Omega|$. For any $p \in [1, \infty]$, we use $c$ to denote a generic constant independent of problem parameters, and $c_\Omega$ to indicate a constant that may depend on the geometry of the domain.

The Lebesgue space $L^p(\Omega)$ consists of all measurable functions $\bfv$ such that
\begin{align*}
\|\bfv\|_{L^p} & =\big( \int_{\Omega} |\bfv (\bfx)| ^p\, d\bfx\big)^{\frac{1}{p}} \, < \infty,  \hspace{1cm} \text{if} \hspace{0.2cm} p \in [1 , \infty),\\
\|\bfv\|_{L^{\infty}}& = \sup_{\bfx \in \Omega}  |\bfv (\bfx)| < \infty,  \hspace{2.3cm} \text{if} \hspace{0.2cm} p  = \infty.
\end{align*}
In particular, the $L^2$ norm and inner product are denoted by $\|\cdot\|$ and $(\cdot, \cdot)$, respectively.

Let $\bfV$ be a Banach space of functions defined on $\Omega$, equipped with the norm $\|\cdot\|_\bfV$. For $p \in [1, \infty]$, we define the Bochner space $L^p(a, b; \bfV)$ as the set of functions $\bfv : (a, b) \rightarrow \bfV$ for which the following norm is finite:
\begin{align*}
\|\bfv\|_{L^p(a, b; B)} &= \left( \int_a^b \|\bfv(t)\|_B^p \, dt \right)^{1/p} < \infty, \qquad \text{for } p \in [1, \infty), \\
\|\bfv\|_{L^\infty(a, b; B)} & = \sup_{t \in (a, b)} \|\bfv(t)\|_B < \infty, \qquad \text{for } p = \infty.
\end{align*}

The space  $W_0^{1,p}(\Omega)$ consists of all functions in
$W^{1,p}(\Omega)$ that vanish on the boundary $\partial \Omega$ (in the sense of traces)
$$ W_0^{1,p}(\Omega) = \big\{ \bfv:   \hspace{0.1cm}\bfv \in  W^{1,p}(\Omega)  \hspace{0.3cm} \text{and} \hspace{0.3cm}\bfv |_{\partial \Omega}=0  \big\}.$$
We introduce the Banach spaces of solenoidal functions
\begin{align*}
L^2_\sigma (\Omega) &= \lbrace \bfv:  \hspace{0.1cm}\bfv \in  L^2(\Omega)\text{,} \hspace{0.3cm} \nabla \cdot \bfv = 0  \hspace{0.3cm} \text{and} \hspace{0.3cm}  \bfv \cdot \bfn |_{\partial \Omega}=0
\rbrace,\\
W_{\sigma}^{1,p}(\Omega) &= \big\{ \bfv:   \hspace{0.1cm}\bfv \in  W^{1,p}(\Omega) \text{,} \hspace{0.3cm} \nabla \cdot \bfv = 0  \hspace{0.3cm} \text{and} \hspace{0.3cm} \bfv |_{\partial \Omega}=0  \big\}, 
\end{align*}
which are equipped with the same norms as $L^2(\Omega)$ and $W^{1,p}_0(\Omega)$, respectively. 
We denote by  $\left( W_\sigma^{1,p}(\Omega) \right)'$ the dual space of $W_{\sigma}^{1,p}(\Omega)$.  We recall the following  inclusions for $p\geq2$
$$
W_{\sigma}^{1,p}(\Omega) \subset L_\sigma^2(\Omega) \subset  \left( W_\sigma^{1,p}(\Omega) \right)'  \hspace{0.3cm} \text{if} \hspace{0.3cm} \frac{2d}{d+2} \leq p < \infty,
$$
where these injections are continuous, dense and   compact. 
For matrix $A = (a_{ij})_{i,j=1}^3$, the Frobenius norm of the matrix $A$ is given by
$$
|A|_F = \left( \sum_{i,j =1}^3 (a_{ij})^2\right)^{\frac{1}{2}}=\left( A:A \right)^\frac12 . 
$$

\subsection*{Inequalities in Banach and Hilbert Spaces}

We recall several standard inequalities in Banach and Hilbert spaces that will be used throughout this work; proofs and further discussion can be found in classical references (see, e.g., \cite{Brezis2011, FMRT01}). 

Let $1 \leq p \leq \infty$ and define the conjugate exponent $p'$ via the relation $\frac{1}{p} + \frac{1}{p'} = 1$. Suppose that $f \in L^p(\Omega)$ and $g \in L^{p'}(\Omega)$. Then the Hölder inequality states:
\begin{equation}\label{Holder}
 \tag{H\"older inequality}
\|fg\|_{L^1} \leq \|f\|_{L^p} \, \|g\|_{L^{p'}}.
\end{equation} 

For any non-negative real numbers $a, b \geq 0$ and any $\lambda > 0$, the Young inequality gives the estimate:
\begin{equation}\label{Young}
 \tag{Young inequality}
ab \leq \lambda \, a^p + \left( p \lambda \right)^{-\frac{p'}{p}} \frac{1}{p'} \, b^{p'}.
\end{equation}

Let $1 < p < \infty$. Then there exists a constant $c_{\text{P}} > 0$ depending only on $p$ and the domain $\Omega$, such that for all $\bff \in W^{1,p}_0(\Omega)$, the Poincaré inequality holds:
\begin{equation}\label{Poincare}
 \tag{Poincar\'e inequality}
\| \bff \|_{L^p} \leq c_{\text{P}} \| \nabla \bff \|_{L^p}.
\end{equation}

Next, we recall the interpolation inequality for Lebesgue spaces. If $\bff \in L^p(\Omega) \cap L^q(\Omega)$ for $1 \leq p, q \leq \infty$, then for any $r$ satisfying
\[
\frac{1}{r} = \frac{\theta}{p} + \frac{1 - \theta}{q}, \quad \text{for } \theta \in [0,1],
\]
it follows that $\bff \in L^r(\Omega)$ and the following inequality holds:
\begin{equation}\label{Interpolation}
\tag{Lebesgue interpolation inequality}
\| \bff \|_{L^r} \leq \| \bff \|_{L^p}^{\theta} \, \| \bff \|_{L^q}^{1 - \theta}.
\end{equation}

For functions in Sobolev spaces, the Ladyzhenskaya inequality provides the bound:
\begin{equation}\label{Ladyzhenskaya}
 \tag{Ladyzhenskaya’s inequality}
\| \bff \|_{L^4} \leq C_L \| \bff \|^{1/2} \, \| \nabla \bff \|^{1/2}, \qquad \text{for } d = 2,
\end{equation}
valid for all $\bff \in W_0^{1,2}(\Omega)$, with a constant $C_L$ depending only on $\Omega$.

Finally, for functions $\bff \in H^2(\Omega) \cap H_0^1(\Omega)$ in two space dimensions, Agmon's inequality provides the following estimate:
\begin{equation}
\label{Agmon}
\tag{Agmon's inequality}
\|\bff\|_{L^\infty} \leq C \| \bff \|_{L^2}^{1/2} \, \| \nabla \bff \|_{H^2}^{1/2}, \qquad \text{for } d = 2,
\end{equation}
where the constant $C$ depends on the domain.

We now recall several standard estimates for solutions $\bfu$ of the Navier--Stokes system \eqref{NSE}, which will serve as a foundational tool in our forthcoming analysis. These bounds are classical and may be found in the literature (see, e.g., \cite{T77, FMRT01}). Let $\lambda_1$ denote the smallest (positive) eigenvalue of the Stokes operator. In what follows, we characterize the complexity of the system using the two-dimensional Grashof number, defined by
\begin{equation}\label{Grashof}
G \coloneqq \frac{1}{\nu^2 \lambda_1} \|\bff\|.
\end{equation}
 \begin{thm}\label{NSEBounds}
Fix \( T > 0 \), and let \( G \) be the Grashof number as defined in \eqref{Grashof}. Suppose that \( \bfu \) is the solution of \eqref{NSE} corresponding to the initial value \( \bfu_0 \). Then, there exists a time \( t_0 \), depending on \( \|\bfu_0\|_{L^2} \), such that for all \( t \geq t_0 \), we have:
\[
\|\bfu(t)\|_{L^2}^2 \leq 2 \nu^2 G^2,
\]
and
\[
\int_t^{t+T} \|\nabla \bfu(\tau)\|_{L^2}^2 \, d\tau \leq 2 \left(1 + T \nu \lambda_1\right) \nu G^2.
\]
\end{thm}

\section[\texorpdfstring{Global Well-Posedness of System \eqref{DALady}}{}]%
        {Global Well-Posedness}%
\label{S3}

In this section, we establish the a priori estimates, which are crucial for demonstrating well-posedness of system \eqref{DALady}. This argument can be rigorously justified through a Galerkin approximation, followed by taking limits under  certain compactness conditions, as outlined in the Aubin-Lions Lemma (see, e.g., \cite{T77}). Given the established nature of this method, applicable to both the 2D and 3D Navier-Stokes (NS) equations, as well as other turbulence models such as the Smagorinsky model \cite{SM63} and the Ladyzhynskaya model \cite{Lady68}, we have opted to omit the detailed exposition at this point.

\begin{thm}

Assume that $p \geq \frac{5}{2}$,   $\bff \in L^{\infty} \left((0, \infty),  L^2(\Omega)\right)$. Let $\bfv_0 \in L^2_{\sigma}(\Omega)$ and  $c_0\, \mu h^2 \leq \nu$. The continuous data assimilation equations \eqref{DALady} has a unique global weak solution that satisfies for all $T >0$

\begin{equation}\label{apriori}
 \bfv \in L^{\infty}\left((0, T); L^2_{\sigma}(\Omega)\right) \cap L^p\left((0,T); W^{1,p}_{\sigma}(\Omega)\right) \cap W^{1,p'}\left((0,T); (W^{1,p}_{\sigma}(\Omega))' \right).   
\end{equation}

\end{thm}
\begin{proof}
First define $\bff_\mu = \bff + \mu\, I_h(\bfu)$.  As strong solution $\bfu \in C\left( [0,T], H^1\right)$,  consequently  
$$\|I_h (\bfu)\| \leq \|\bfu - I_h(\bfu)\| + \|\bfu\| \leq \left(c_0\, h + \lambda^{-\frac{1}{2}} \right) \|\nabla \bfu\|$$ 
implies that  $\bff_\mu \in C\left( [0,T], L^2\right)$. This means that there is a constant $M$ depends on $h, \lambda_1,  \|\bfu\|_{H^1}$ and $\|\bff\|$ such that $\|\bff_\mu\|^2<M$ for every $t \in [0,T ]$.  Then take a formal inner-product of  \eqref{DALady} with $\bfv$ and integrating by parts,  we arrive at  

\begin{equation}\label{EnergyEq1}
\frac{1}{2} \partial_t \|\bfv\|^2+ \nu \|\nabla \bfv\|^2 +\bar{\nu} \|\nabla \bfv\|^p_{L^p} = \left( \bff_\mu,  \bfv\right)  - \mu \, \left( I_h(\bfv), \bfv\right).
\end{equation}

The first term on the right side of the above is estimated as 
$$\left( \bff, \bfv\right) \leq \|\bff_\mu\| \|\bfv\| \leq \frac{1}{2\mu} \|\bff_\mu\|^2 + \frac{\mu}{2} \|\bfv\|^2.$$

Using \eqref{I_h},  we next estimate the second them on the right hand side of  \eqref{EnergyEq1} as 

\begin{equation*}
\begin{split}
- \mu \, \left( I_h(\bfv), \bfv\right) = \mu \, \left(\bfv -  I_h(\bfv), \bfv\right) - \mu \|\bfv\|^2  &\leq \frac{\mu}{2} \|\bfv -  I_h(\bfv)\|^2 + \frac{\mu}{2}  \|\bfv\|^2 - \mu \|\bfv\|^2\\
& \leq \frac{c_0 \mu\, h^2}{2} \|\nabla \bfv\|^2 - \frac{\mu}{2}  \|\bfv\|^2. 
\end{split}
\end{equation*}
By hypothesis, the size of  $h$ is chosen to be small enough such that $c_0\, \mu h^2 \leq \nu$.   Therefore

\begin{equation}\label{EnergyEq2}
 \partial_t \|\bfv\|^2+ \nu \|\nabla \bfv\|^2 +\bar{\nu} \|\nabla \bfv\|^p_{L^p} \leq \frac{1}{\mu} \|\bff_\mu\|^2, 
\end{equation}
and consequently
$$
 \partial_t \|\bfv\|^2+ \nu \lambda_1 \| \bfv\|^2 +\bar{\nu} \|\nabla \bfv\|^p_{L^p} \leq \frac{1}{\mu} \|\bff_\mu\|^2.
$$
Therefore, after dropping the $L^p$ term,  Gr\"onwall inequality implies 
$$\|\bfv(t)\|^2 \leq \|\bfv_0\|^2 e^{-\nu \, \lambda_1\, t} + \frac{M}{\mu\, \nu\, \lambda_1} \left( 1 - e^{-\nu \, \lambda_1\, t} \right) \leq K_1,$$
for every $t \in  [0,T)$, where 
$$K_1 = \|\bfv_0\|^2 + \frac{M}{\mu\, \nu\, \lambda_1}.$$
Therefore,
\begin{equation}
 \bfv \in L^{\infty}\left((0, T); L^2(\Omega)\right).   
\end{equation}

Now, integrating \eqref{EnergyEq2} yields, 
\begin{equation}
    \int_0^t \|\nabla \bfv\|^2 d\tau \leq K_2,  \hspace{0.3cm} \mbox{and} \hspace{0.3cm}\int_0^t \|\nabla \bfv\|^p_{L^p} d\tau \leq K_3,
    \end{equation}
for every $t \in  [0,T)$, where 
$$K_1= K_1(\bfv_0, \nu, \bff_\mu, T) \hspace{0.3cm} \mbox{and} \hspace{0.3cm}  K_2= K_2(\bfv_0, \bar{\nu}, \bff_\mu, T).$$

Next we investigate the time derivative $\partial_t \bfv$.  For any test function  $\varphi \in W_\sigma^{1,p}(\Omega)$, 
\begin{equation*}
\begin{split}
\lb \partial_t \bfv, \varphi \rb+ \int_{\Omega}  \bfv \cdot \nabla \bfv \cdot\, \varphi\, d\bfx + \int_{\Omega}  2\nu \, \nabla  \bfv : \nabla   \varphi\, &+ \,  \bar{\nu} \, |\nabla  \bfv|^{p-2}_F \,  \nabla  \bfv : \nabla  \varphi \, d\bfx \\
& =  \int_{\Omega}  \bff_\mu \cdot \varphi \, d\bfx - \mu \int_{\Omega}  I_h \bfv\cdot \varphi \,d\bfx, 
\end{split}
\end{equation*}
for almost all $t \in [0,T]$.  Due to the incompressiblity condition, the nonlinear term can be written as $ \bfv \cdot \nabla \bfv = \nabla \cdot \left( \bfv \otimes  \bfv\right)$. Then, we have
\begin{equation*}
\begin{split}
\left| \lb \partial_t \bfv ,  \varphi\, \rb  \right|  \leq  \left| \int_{\Omega}  \bfv   \otimes \bfv : \nabla \varphi\, d\bfx \right | & + \left| \int_{\Omega}  \nu \, \nabla  \bfv : \nabla \varphi\,  + \,  \bar{\nu} \, |\nabla  \bfv|^{p-2}_F\,\nabla  \bfv : \nabla \varphi \, d\bfx \right|\\
&+ \left|  \int_{\Omega}  \bff_\mu \cdot \varphi \, d\bfx \right| +  \mu  \left|\int_{\Omega}  I_h \bfv\cdot \varphi \,d\bfx \right|.
\end{split}
\end{equation*}
Let $p' = \frac{p}{p-1}$, and note that $p' < p$ for $p\geq \frac{5}{2}$.   Using the \ref{Holder} along with \eqref{I_h} yields
\begin{equation*}
\begin{split}
\left| \lb \partial_t \bfv ,  \varphi\, \rb  \right|  \leq \|\bfv\|_{L^{2p'}}^2\, \|  \nabla \varphi\|_{L^p} + \nu \|  \nabla  \bfv\|_{L^{p'}}\, \| \nabla \varphi \|_{L^p} + \bar{\nu} \|  \nabla  \bfv\|^{p-1}_{L^p} \, \|  \nabla \varphi\|_{L^p} + \| \bff_\mu \|  \, \| \varphi \|  +  \mu c_I\,  \| \bfv \|  \, \| \varphi \|.
\end{split}
\end{equation*}
By taking supremum of the above inequality over all  $\varphi \in W_\sigma^{1,p}(\Omega)$   such that  $\| \varphi\|_{W_\sigma^{1,p}(\Omega)} = 1$,  and using the \ref{Interpolation}, we obtain
\begin{equation*}
\begin{split}
\|\partial_t \bfv\|_{\left(W_\sigma^{1,p}(\Omega)\right)'} &
\leq  \|\bfv\|^2_{L^{2p'}} \, + \nu \|  \nabla  \bfv\|_{L^{p'}}\,  + \bar{\nu} \|  \nabla  \bfv\|^{p-1}_{L^p} \, + C \, \| \bff_\mu \|   + \mu \,  c_I \, C \| \bfv\| \\
& \leq  C\,  \|\bfv\|^{\frac{2p-3}{p}}\,  \| \nabla \bfv\|^{\frac{3}{p}}+ 
\nu \, C \,  \| \nabla \bfv\|_{L^2}  +  \bar{\nu} \, C \|  \nabla \bfv\|^{p-1}_{L^p} \, + C \, \| \bff_\mu \|   + \mu \,  c_I \, C \, \| \bfv\|\\
&  \leq  C \, \|\bfv\|^{\frac{2p-3}{p}}\,  \| \nabla \bfv\|_{L^p}^{\frac{3}{p}} + \nu \, C \,  \| \nabla \bfv\|_{L^p}  +  \bar{\nu} \, C \, \|  \nabla \bfv\|^{p-1}_{L^p} \, + C \, \| \bff_\mu \|   + \mu \,  c_I \, C \, \| \bfv\|,
\end{split}
\end{equation*}
where $C$ only depends on $p$ and $\Omega$.  Hence, 
\begin{equation*}
\begin{split}
& \|\partial_t \bfv\|^{p'}_ {L^{p'} \left( 0 , T ; \,  \left(W_\sigma^{1,p}(\Omega)\right)' \right)} = \int_0^T \|\partial_t \bfv(\tau)\|^{p'}_{\left(W_\sigma^{1,p}(\Omega) \right)'}\, d \tau \\
 & \leq C \int_0^T  \|\bfv(\tau)\|^{\frac{2p-3}{p-1}}\,  \|\nabla\bfv(\tau)\|_{L^p}^{\frac{3}{p-1}} \, d\tau  + \nu^{p'} \, C \int_0^T \|\nabla \bfv (\tau)\|_{L^p}^{\frac{p}{p-1}}\, d \tau +  \bar{\nu}^{p'} C \int_0^T\|\nabla \bfv(\tau)\|_{L^p}^{p} \, d \tau\\
 & \quad + \, C \int_0^T \, \| \bff_\mu(\tau) \|^{p'}\, d\tau + \mu^{p'} \,  c_I^{p'} \, C \int_0^T \| \bfv(\tau)\|^{p'} \,dt\\
 &\leq C \, \|\bfv\|^{\frac{2p-3}{p-1}}_{L^\infty \left(0, T; L^2_\sigma(\Omega)\right)} \, T^{\frac{1}{\alpha}} \,  \|\bfv\|_{L^p\left(0, T; W^{1,p}_\sigma(\Omega)\right)}^{\frac{3}{p-1}}  \, + \,  \nu^{p'} \, C\, T^{\frac{1}{\beta}} \, \|\bfv\|_{L^p\left(0, T; W^{1,p}_\sigma(\Omega)\right)}^{p'} \, + \,  \bar{\nu}^{p'} \,C\, \|\bfv\|_{L^p\left(0, T; W^{1,p}_\sigma(\Omega)\right)}^p\\
 & \quad + \,  C \, T^{\frac{1}{\gamma}} \, \|\bff_\mu\|^{p'}_{L^2\left(0, T; L^2_\sigma(\Omega)\right)}  + \mu^{p'} \,  c_I^{p'} \, C \,  \, T^{\frac{1}{\gamma}} \, \|\bfv\|^{p'}_{L^2\left(0, T; L^2_\sigma(\Omega)\right)}.
\end{split}
\end{equation*}
where $\alpha,  \beta $ and $\gamma$ are the conjugate exponents to $(p-1)p/3, p-1 $  and $2/p'$, respectively, and the constant $C$ depends only on $p$ and $\Omega$. Therefore,  we have
\begin{equation}
\label{c3tilde}
\|\partial_t \bfv\|_ {L^{p'} \left( 0 , T ; \,  \left(W_\sigma^{1,p}(\Omega)\right)' \right)} \leq K_3,
\end{equation}
where 
$$K_3=K_3(\nu, \bar{\nu}, \bff_\mu,  T).$$

With the above estimates established, it is possible to prove uniqueness and also to apply the Galerkin approximation scheme to demonstrate the existence of weak solutions through standard arguments.
\end{proof}

\begin{re}
With the above a priori bounds established in \eqref{apriori}, the well-posedness argument can alternatively be approached using the Schauder fixed-point theorem. We conjecture that, in this case, the small data condition \( c_0 \mu h^2 \leq \nu \) may not be necessary. For further details, we direct interested readers to \cite{Cao_Giorgini_Jolly_Pakzad_2022}.
 
\end{re}

\section{Synchronization}\label{S4}
   
 \begin{thm}\label{DAsynch}
 For $p \geq \frac{5}{2}$, let $\bff \in L^2$ and let $\bfu$ be a strong solution of \eqref{NSE} with with periodic boundary conditions. Let $\bfv$ be the solution to the data assimilation algorithm given by \eqref{DALady}. Then for $\mu$ large enough and $h$ small enough such that  
  $$ \mu \geq 8 \nu \lambda_1 G^2 \hspace{0.5cm} \text{and} \hspace{0.5cm}  2\, \mu \, c_0 \, h^2 \leq \nu, $$
  we have 
  \begin{equation*}
\label{}
 \begin{split}
\|\bfu - \bfv\|^2 \leq  C_{\bfu, \Omega}\, \bar{\nu} (1 - e^{-t}) + \|\bfu(0) - \bfv(0)\|\, e^{-t},
\end{split}
\end{equation*}
where   the constant $C_{\bfu, \Omega}$ depends on $\|\bfu\|_{H^1}, \|\bfu\|_{H^3}$ and size of the domain $|\Omega|$.  In particular, 
 \begin{equation*}
\label{}
 \begin{split}
\|\bfu - \bfv\|^2 \leq  C_{\bfu, \Omega}\, \bar{\nu},
\end{split}
\end{equation*}
as $t \rightarrow \infty$. 
 \end{thm}\label{L2Convergence}

\begin{proof}
     Subtracting \eqref{NSE} and \eqref{DALady},  and  after employing 
 $$(\bfu \cdot \nabla) \bfu - (\bfv \cdot \nabla) \bfv = (\bfw \cdot \nabla) \bfu  + (\bfv \cdot \nabla) \bfw, $$
 the difference  $\bfw= \bfu - \bfv$  satisfies the following error equation
 \begin{equation*}
 \begin{split}
\partial_t \bfw+   (\bfw \cdot \nabla) \bfu  + (\bfv \cdot \nabla) \bfw  - \nu \Delta \bfw & +  \nabla \cdot \left( \bar{\nu} \, |\nabla \bfv|^{p-2}\, \nabla \bfv\right) +  \nabla (p - q)   = -  \mu\, I_h(\bfw),\\ 
\nabla \cdot \bfw  &= 0. 
\end{split}
\end{equation*}
Multiplying the above by $\bfw$ in $L^2(\Omega) $ space,  after integrating by part and using $\nabla \cdot \bfw=0$,  we arrive at the following energy type of equation
\begin{equation*}
\frac{1}{2} \frac{d}{dt} \|\bfw\|^2 + (\bfw \cdot \nabla \bfu, \bfw)+ \nu \|\nabla \bfw\|^2= -  \mu\,(I_h(\bfw), \bfw) +  \left( \bar{\nu} \, |\nabla \bfv|^{p-2}\, \nabla \bfv,   \nabla \bfw \right). 
\end{equation*}
Since the true equation \eqref{NSE} lacks the p-Laplacian term, we cannot leverage its monotonicity properties. As a result, the above  error equation includes a nonlinear term that depends  on the nudging solution $\bfv$. To proceed, we express $\bfv$ as $\bfv = \bfu - \bfw$, introducing four additional nonlinear p-Laplacian-type terms. Each term requires careful estimation, forming a central aspect of the proof.  In light of $|\nabla \bfv|^{p-2} = |\nabla \bfu - \nabla \bfw|^{p-2} \leq 2^{p-2} \left(|\nabla \bfu|^{p-2} + |\nabla \bfw|^{p-2} \right)$, the non linear p-Laplacian term in the above can be estimated as 

$$  \left(|\nabla \bfv|^{p-2}\, \nabla \bfv,   \nabla \bfw \right) \leq  2^{p-2} \Big[\left(|\nabla \bfu|^{p-2} + |\nabla \bfw|^{p-2} \right) (\nabla \bfu - \nabla \bfw), \nabla \bfw \Big].$$
Thus we obtain 
\begin{equation}
\label{Energyeq}
 \begin{split}
\frac{1}{2} \frac{d}{dt} \|\bfw\|^2 + \nu \|\nabla \bfw\|^2 +  2^{p-2} \, \bar{\nu} \|\nabla \bfw\|_{L^p}^p  & \leq    - (\bfw \cdot \nabla \bfu, \bfw)-  \mu\,(I_h(\bfw), \bfw)\\
&  +  2^{p-2} \bar{\nu}  \left(|\nabla \bfu|^{p-2} \nabla \bfu , \nabla \bfw \right)\\& +  2^{p-2} \bar{\nu}  \left(|\nabla \bfw|^{p-2} \nabla \bfu , \nabla \bfw \right) \\
& -  2^{p-2} \bar{\nu}  \left(|\nabla \bfu|^{p-2} \nabla \bfw , \nabla \bfw \right)\\
& \coloneqq \RN{1} + \RN{2} + \RN{3} + \RN{4} - \RN{5}. 
\end{split}
\end{equation}
Next we have the following estimates on the above terms. 

\subsection*{Term \RN{1}} Using \ref{Holder} along with   \ref{Ladyzhenskaya}  give  
\begin{equation*}
 \begin{split}
|(\bfw \cdot \nabla \bfu, \bfw)| \leq \|\bfw\|_{L^4}^2 \|\nabla \bfu\| \leq  \|\bfw\|  \|\nabla \bfw\|  \|\nabla \bfu\| \leq \frac{1}{\nu} \|\nabla \bfu\|^2 \|\bfw\|^2 + \frac{\nu}{4} \|\nabla \bfw\|^2. 
\end{split}
\end{equation*}

\subsection*{Term \RN{2}}  Applying \ref{Holder} and   \eqref{I_h} and after  using the assumption $2\, \mu \, c_0 \, h^2 \leq \nu$, we obtain 
\begin{equation*}
\begin{split}
-  \mu\,(I_h(\bfw), \bfw) &= \mu (\bfw - I_h(\bfw), \bfw) - \mu \|\bfw\|^2 \leq  \mu \, c_0 \, h \|\nabla \bfw\| \|\bfw\| - \mu \|\bfw\|^2\\
& \leq \frac{1}{2}\,  \mu\,  c_0^2\,  h^2   \|\nabla \bfw\|^2 - \frac{1}{2} \mu  \|\bfw\|^2 \leq  \frac{\nu}{4} \|\nabla \bfw\|^2 - \frac{1}{2} \mu  \|\bfw\|^2. 
\end{split}
\end{equation*}

\subsection*{Term \RN{3}} With $c_{\Omega}$ being a generic constant depending on the size of domain,  using  \ref{Holder} along with  embedding of $L^p$ spaces gives 
\begin{equation*}
\begin{split}
 \bar{\nu}\,  \left(|\nabla \bfu|^{p-2} \nabla \bfu , \nabla \bfw \right) &\leq    \bar{\nu}\, \|\nabla \bfu\|_{L^\infty}^{p-1}\, \|\nabla \bfw\|_{L^1} \leq    c_{\Omega}\,  \bar{\nu}\,  \|\nabla \bfu\|_{L^\infty}^{p-1}\,  \|\nabla \bfw\|_{L^p}\\
 &\leq c_{\Omega}\,    \bar{\nu}\,  \|\nabla \bfu\|_{L^\infty}^p +    \frac{\bar{\nu}}{6} \, \|\nabla \bfw\|_{L^p}^p\\
 &  \leq c_{\Omega}\,    \bar{\nu}\,  \| \bfu\|_{H^1}^p \, \| \bfu\|_{H^3}^p  +    \frac{\bar{\nu}}{6} \, \|\nabla \bfw\|_{L^p}^p.
\end{split}
\end{equation*}
\subsection*{Term \RN{4}} Using \ref{Agmon},  the other term is estimated as 
\begin{equation*}
\begin{split}
\bar{\nu} \, \left(|\nabla \bfw|^{p-2} \nabla \bfu , \nabla \bfw \right)&  \leq \bar{\nu}\,  \| \nabla \bfu\|_{L^\infty}\,  \| \nabla \bfw\|_{L^{p-1}}^{p-1} \leq \bar{\nu}\,  \|\bfu\|_{H^1}^{\frac{1}{2}}\, \|\bfu\|_{H^3}^{\frac{1}{2}}\,   \| \nabla \bfw\|_{L^{p}}^{p-1}\\
& \leq c \, \bar{\nu} \|\bfu\|_{H^1}^{\frac{p}{2}}\, \|\bfu\|_{H^3}^{\frac{p}{2}} + \frac{\bar{\nu} }{6} \| \nabla \bfw\|_{L^{p}}^{p}.
\end{split}
\end{equation*}
\subsection*{Term \RN{5}}  Similarly 
\begin{equation*}
\begin{split}
\bar{\nu}\,  \left(|\nabla \bfu|^{p-2} \nabla \bfw , \nabla \bfw \right) &  \leq \bar{\nu}\,  \| \nabla \bfu\|_{L^\infty}^{p-2}\,  \| \nabla \bfw\|_{L^{2}}^{2} \leq \bar{\nu}\,  \|\bfu\|_{H^1}^{\frac{p-2}{2}}\, \|\bfu\|_{H^3}^{\frac{p-2}{2}}\,   \| \nabla \bfw\|_{L^{p}}^{2}\\
& \leq c \, \bar{\nu} \|\bfu\|_{H^1}^{\frac{p}{2}}\, \|\bfu\|_{H^3}^{\frac{p}{2}} + \frac{\bar{\nu} }{6} \| \nabla \bfw\|_{L^{p}}^{p}.
\end{split}
\end{equation*}

Combining all these above bounds into the terms in \eqref{Energyeq},   we obtain 

\begin{equation}
\label{}
 \begin{split}
\frac{d}{dt} \|\bfw\|^2 + \nu \|\nabla \bfw\|^2+ \bar{\nu}\, \| \nabla \bfw\|_{L^{p}}^{p} \leq C_{\bfu, \Omega}\, \bar{\nu} + \left( \frac{2}{\nu} \|\nabla \bfu\|^2 - \mu \right)\, \|\bfw\|^2, 
\end{split}
\end{equation}
where the constant $C_{\bfu, \Omega}$ depends on $\|\bfu\|_{H^1}, \|\bfu\|_{H^3}$ and size of the domain $|\Omega|$.  Hence
\begin{equation*}
\label{}
 \begin{split}
\frac{d}{dt} \|\bfw\|^2 +  \left( \mu - \frac{2}{\nu} \|\nabla \bfu\|^2\right)\, \|\bfw\|^2 \leq C_{\bfu, \Omega}\, \bar{\nu}.
\end{split}
\end{equation*}
By setting \( T = \frac{1}{\nu \lambda_1} \) and considering \( t \geq t_0 \) in Theorem \ref{NSEBounds}, one can assume \( \mu \geq 8 \nu \lambda_1 G^2 \) to obtain the following estimate, utilizing a standard Gr\"onwall inequality

\begin{equation*}
\label{}
 \begin{split}
\|\bfw\|^2 \leq  C_{\bfu, \Omega}\, \bar{\nu} (1 - e^{-t}) + \|\bfw(0)\|\, e^{-t},
\end{split}
\end{equation*}
for all $t \geq t_0$. 
 \end{proof}

\section{Numerical Illustration}\label{S5}

\subsection{2D Simulations}
Motivated by  the experiment conducted by G.I. Taylor and M.M.A. Couette \cite{F95},  we examine a two-dimensional cross-section of the flow between two rotating cylinders.  We use the first order scheme Algorithm \ref{Algorithm}, while the P2-P1 Taylor-Hood mixed finite elements are utilized for discretization in space \cite{L08}. We focus on the case \( p = 3 \) in \eqref{DALady}, corresponding to the Smagorinsky model~\cite{SM63}, and, more specifically, to its dynamic variant introduced by Germano et al.~\cite{GPMC91}. The algorithm is implemented by using public domain finite element software FreeFEM++ \cite{FreeFEM}.  In what follows,  $H$ stands for the DNS mesh  and $h$ represents a coarse spatial resolution where the observational data are collected, i.e. $H << h$.

\begin{Alg}\label{Algorithm} (First order linearly implicit, Backward-Euler scheme)
  Given body force $\bff \in L^{\infty} ((0, \infty), L^2(\Omega))$,  initial condition $\bfv_0 \in \bfV_H$ , reference solution $\bfu^H_{n+1}$, and  $(\bfv^H_n, q_n^H) \in \bfX^H \times Q^H$,  compute  $(\bfv_{n+1}^H, q_{n+1}^H) \in \bfX^H \times Q^H$ satisfying
\begin{multline}\label{IMEX-FEM}
(\frac{\bfv_{n+1}^{H} - \bfv_{n}^{H}}{\Delta t},\Theta ^{H}) + b(\bfv^H_n,\bfv_{n+1}^{H},\Theta^{H}) + \nu \, (\nabla \bfv_{n+1}^{H},\nabla \Theta^{H}) + (\bar{\nu}\, |\nabla \bfv^H_{n}|\, \nabla \bfv^H_{n+1},   \nabla \Theta^{H})  - (q_{n+1}^{H}, \nabla \cdot \Theta^{H})
\\ =  (\bff_{n+1},\Theta^{H})  - \mu \, (I_h(\bfv^H_{n+1} - \bfu_{n+1}), \Theta^h),
\end{multline}
\begin{equation}\label{incomp}
(\nabla \cdot \bfv_{n+1}^{H},r^{H}) = 0,
\end{equation}
for all $(\Theta^{H},  r^{H}) \in  \bfX^H \times Q^H.$
\end{Alg}
The above Algorithm  leads to a linear problem for a continuous and coercive operator for $\bfv^{n+1}$. Existence of $\bfv^{n+1}$  follows from the Lax-Milgram lemma.  Choose the  annulus  domain with  a disk obstacle inside it, see Figure \ref{Domain},  
$$\Omega = \{ (x,y) \in \mathbb{R}^2: x^2 + y^2 \leq 1^2 \hspace{0.3cm} \text{and}  \hspace{0.3cm} x^2 + y^2 \geq (0.1)^2 \}.$$
\begin{figure}[!ht]
  \centering
  \begin{subfigure}[b]{0.6\linewidth}
    \includegraphics[width=\linewidth]{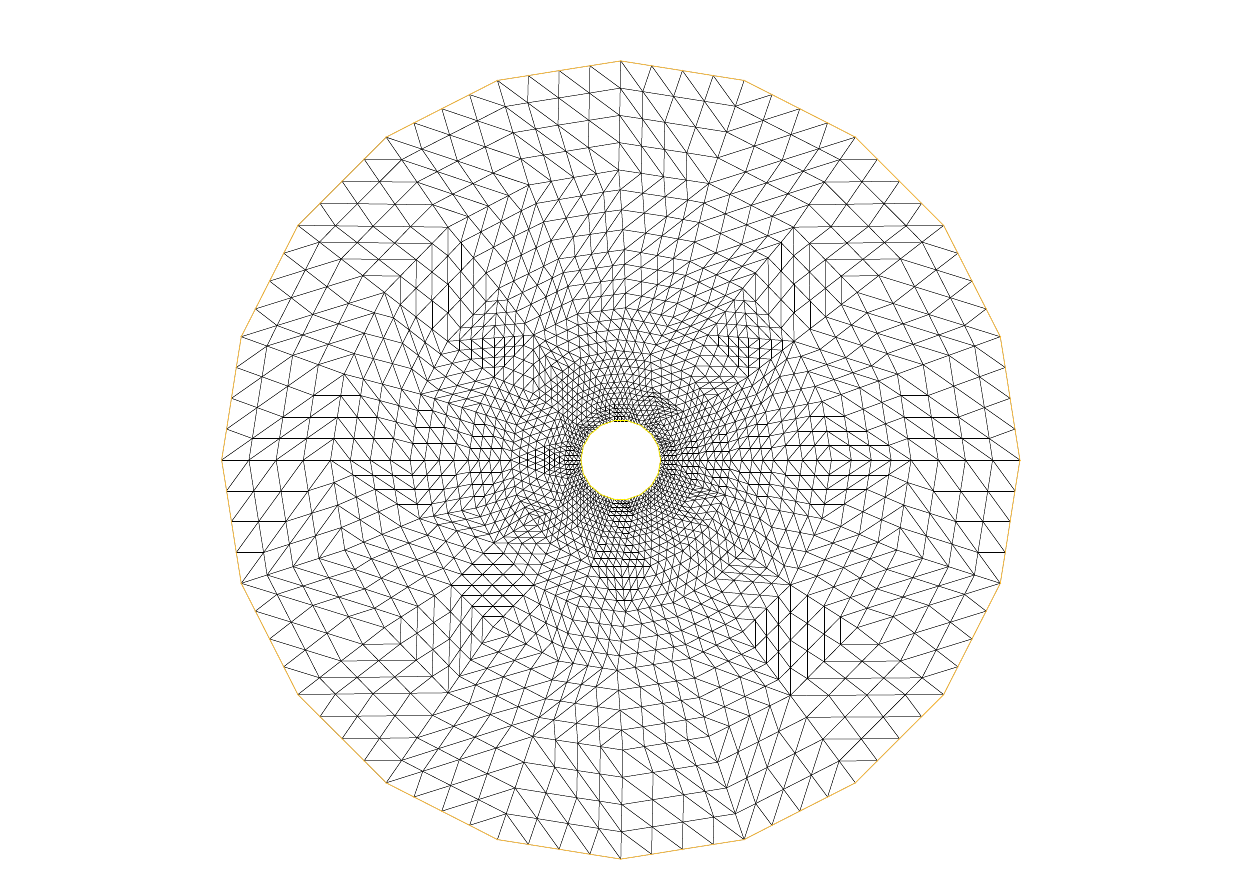}
  \end{subfigure}
 \caption{Domain $\Omega$, and the corresponding coarse spatial resolution $h$ where the data are collected. }
   \label{Domain}
   \end{figure}
 We compute the problem on a Delaunay-Vornoi generated triangular mesh. The number of degrees of freedom is 9432,  the shortest triangular edge is  0.00644635, and the longest edge is 0.0481415.   In the absence of the internal body force,  the rotational force at the outer circle drives the flow, while the inner circle is subject to no-slip boundary conditions.    We take  final time $T = 100$, time step $\Delta t = 0.01$ and Reynolds number $\Rey= 600$.   For the interpolant, we use linear  interpolation
on a mesh that is one refinement coarser, i.e.  on the Delaunay mesh without the barycenter
refinement shown in Figure \ref{Domain}.   As suggested in \cite{SM63}, the parameter  $\bar{\nu}$ is chosen  from dimensional considerations to be 
$$\bar{\nu} = (C_s\, \delta)^2,$$
where   $C_s=0.1$ \cite[Page 192]{L08}  is the standard model parameter, and  $\delta$ is the turbulence-resolution length scale (associated with the mesh size), which, as suggested in the Germano dynamic model \cite{GPMC91}, varies dynamically within the domain, $\delta = H(x,y)$. 

Using Algorithm \ref{Algorithm} with an arbitrary initial condition, we compute the approximating solution $\bfv^H$.   As a true solution is not available for this problem, we utilize the DNS solution as our reference solution, namely $\bfu^H$. At the time step $t_{n+1}$, our observations of the system at a coarse spatial resolution are denoted by $I_h (\bfu_{n+1}^H)$. It should be noted that at this time, we have already obtained the necessary data about the true solution. Therefore, we interpret $I_h(\bfu_{n+1}^H)$ as the most recent available data.

Before presenting the main results, we begin by showcasing numerical illustrations of  both nudged Navier-Stokes and Ladyzhynskaya equations using observational data from the  corresponding models.   Figure \ref{SameData} displays the time evolution of $L^2$ norms for the velocity difference between the reference and nudged solutions.   It demonstrates that both models achieve synchronization up to machine precision, consistent with the results reported in \cite{AOT14} for the NSE and in \cite{Cao_Giorgini_Jolly_Pakzad_2022} for the Ladyzhynskaya model.
\begin{figure}[!ht]
  \centering
  \begin{subfigure}[b]{0.49\linewidth}
    \includegraphics[width=\linewidth]{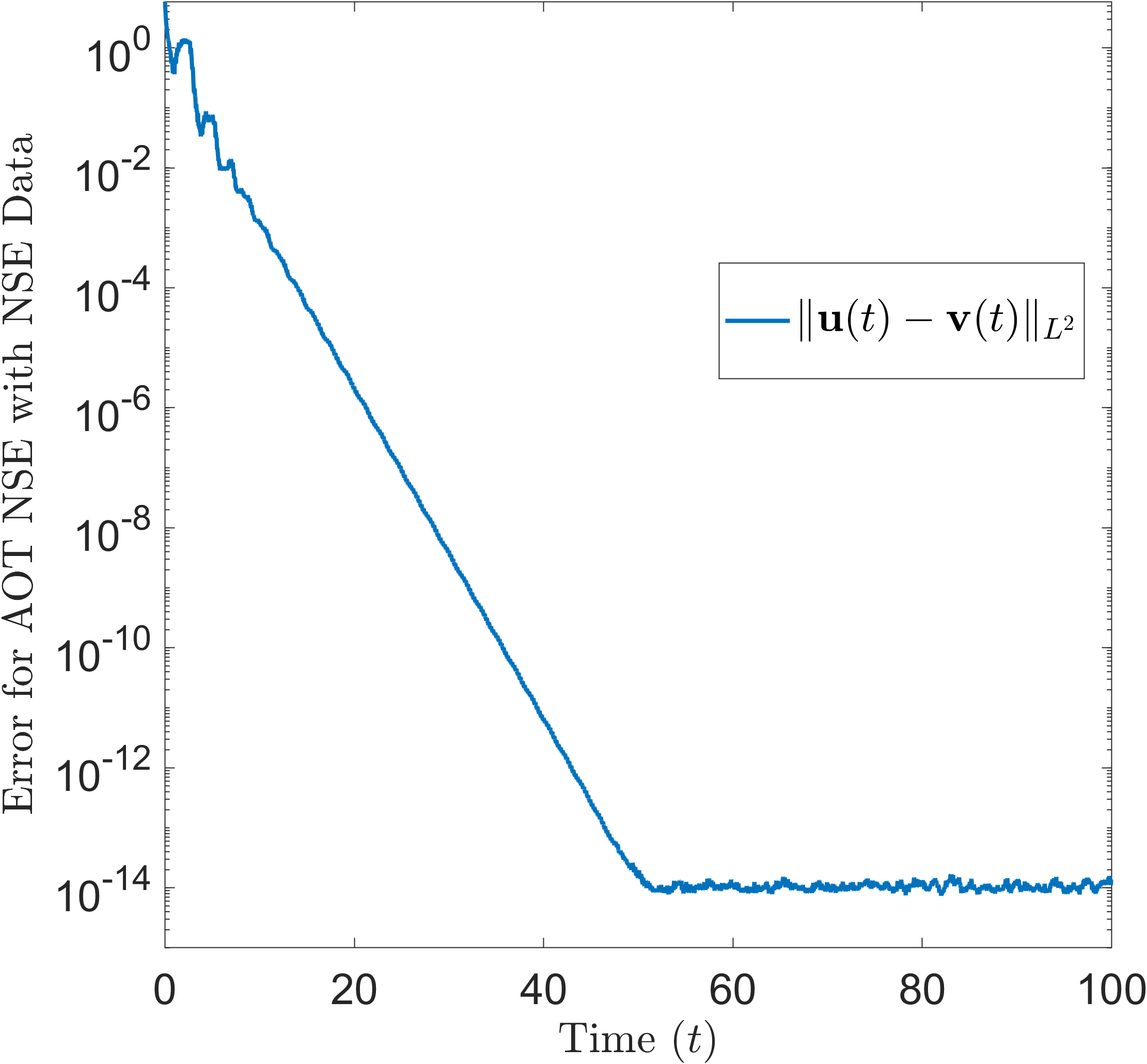}
  \end{subfigure}
 \begin{subfigure}[b]{0.49\linewidth}
    \includegraphics[width=\linewidth]{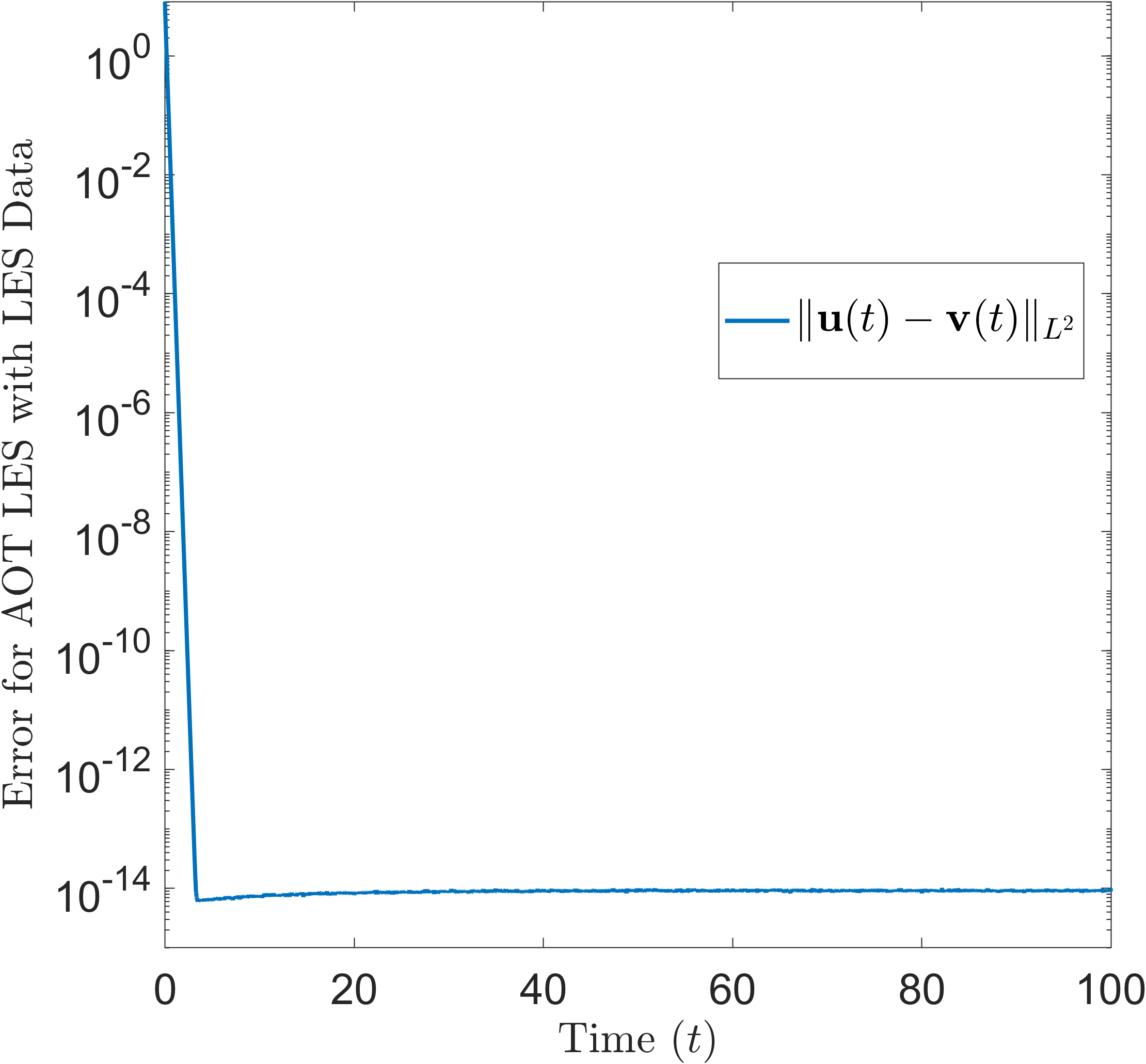}
 \end{subfigure}
 \caption{Left:  $L^2$ error for AOT applied to the NSE  with observed data from the NSE. Right:  $L^2$ error for AOT applied to the Ladyzhenskaya model with observed data from the Ladyzhenskaya model.}
   \label{SameData}
   \end{figure}

The DNS is  run from $t=-5$ to $t_0=0$ to generate the initial condition $\bfu_0$.  In order to calculate the nudged solution for equations \eqref{DALady}, we utilize zero initial conditions where the velocity, $\bfv_0= \vec{0}$.  We use a spatial and temporal discretization setup  introduced in Algorithm  \ref{Algorithm} that is identical to that of the DNS, with $\mu=10$, and then we begin the nudging process by incorporating the DNS solution.  The time progression of the $L^2$ norms for the velocity difference between the NSE \eqref{NSE} solution  and nudged Smagorinsky model \eqref{DALady}  is shown in Figure \ref{SMwNSE}  for different scenarios.    In order to confirm the result derived in Theorem \ref{L2Convergence}, we have considered various values (including many that are not physically realistic) for the model parameter $C_s$, which allows us to vary the viscosity of the model according to the relationship $\bar{\nu} = (C_s \delta)^2$.

\begin{figure}[!ht]
  \centering
  \begin{subfigure}[b]{0.49\linewidth}
    \includegraphics[width=\linewidth]{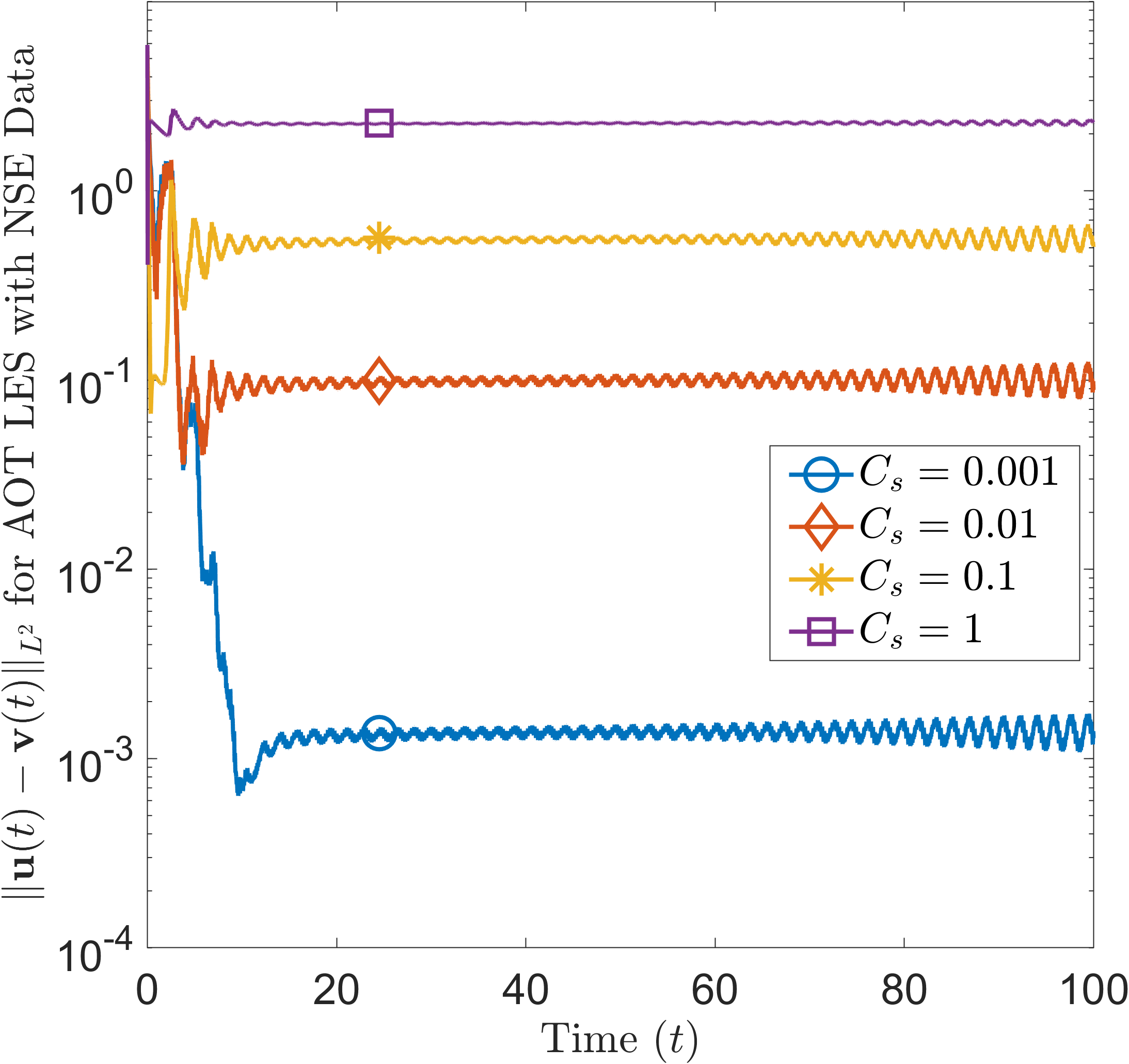}
   \end{subfigure}
 \begin{subfigure}[b]{0.49\linewidth}
    \includegraphics[width=\linewidth]{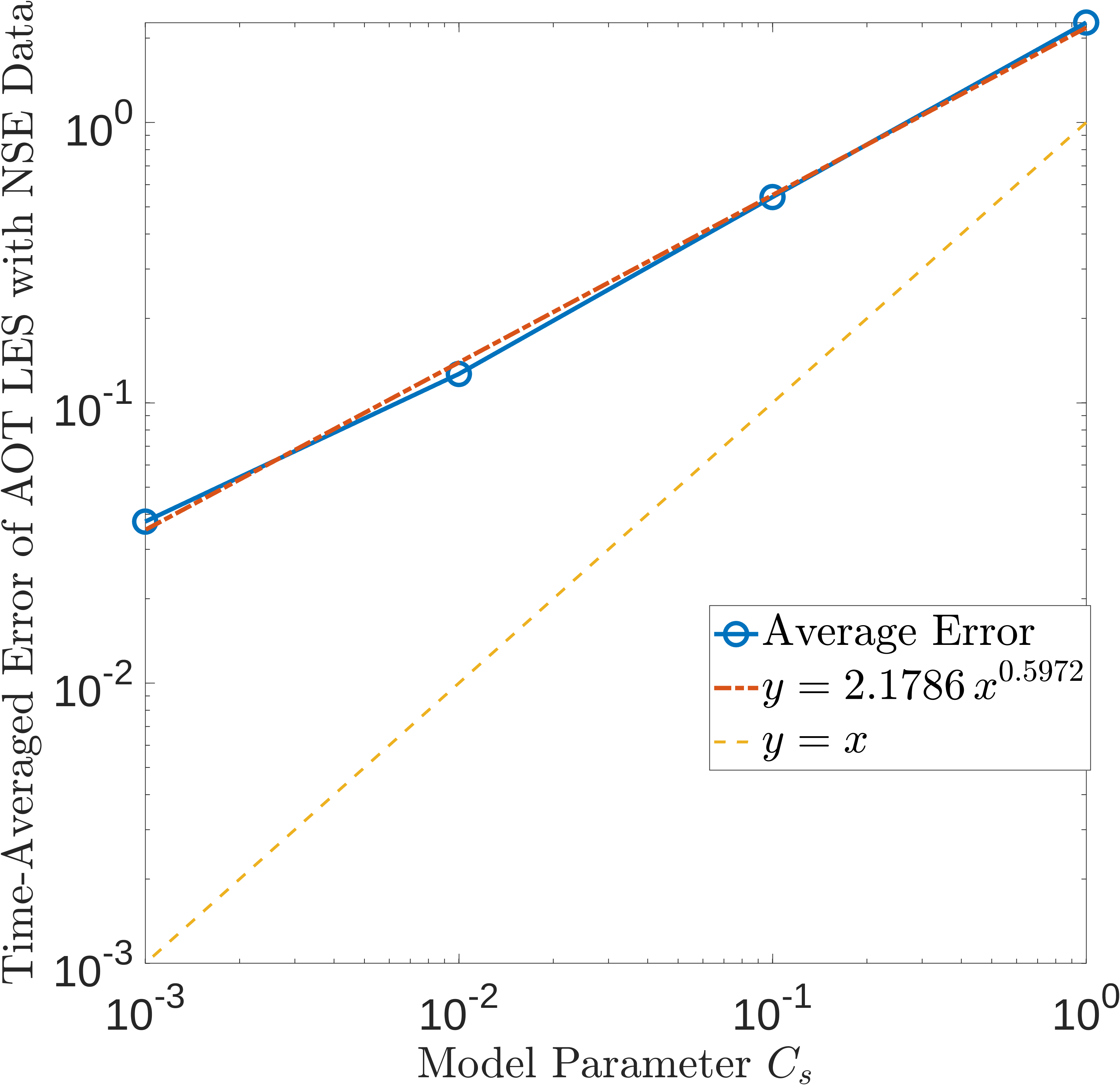}
   \end{subfigure}
 \caption{Left:  $L^2$  error for AOT applied to the Ladyzhenskaya model with observed data from the NSE with different value for the turbulence viscosity parameter $\bar{\nu}= (C_s\delta)^2$. Right: Time-averaged $L^2$ errors versus  the model parameter (least-squares fit included as red $\cdot$- plot for reference).}
   \label{SMwNSE}
\end{figure}

\subsection{3D Simulations}
In this subsection, we present some results from 3D simulations. Note that our analysis was performed in the 2D case rather than the 3D case due to the usual obstructions with the 3D Navier-Stokes equations (although see, e.g., \cite{Biswas_Price_2020_AOT3D} for the AOT algorithm in the 3D case).  However, these analytical obstructions do not prevent us from applying the algorithm in 3D simulations.  First, as a proof of concept, we demonstrate the algorithm's performance in the setting of the 3D Navier-Stokes equations with the Ladyzhenskaya-Smagorinsky LES model.  We note that AOT has been investigated computationally in the 3D case in 
\cite{DiLeoni_Clark_Mazzino_Biferale_2018_inferring,DiLeoni_Clark_Mazzino_Biferale_2019,Larios_Pei_Victor_2023_second_best}.

In this section, we use a fully dealiased (according to the 2/3's dealiasing law) spectral method on a periodic domain $[0,1]^3$. We fix a spatial resolution of $256^3$ grid points. This choice of fairly low resolution is due to the time complexity of employing the Ladyzhenskaya model with Fourier methods. We consider the case $ p =3$, specifically the Smagorinsky model \cite{SM63}  which is a common turbulence model used in large eddy simulation \cite{J04, P04}. In each simulation, we employ Taylor-Green forcing: 
\begin{empheq}[left=\empheqlbrace]{align*}
    f_1(x,y,z) &= \sin (2 \pi x) \cos(2 \pi y) \cos(2 \pi z)\,,\\
    f_2(x,y,z) &= -\cos (2 \pi x) \sin (2 \pi y) \cos(2 \pi z)\,,\\
    f_3(x,y,z) &= 0\,.
\end{empheq}
We begin each simulation from data attained from a simulation with isotropic randomly distributed initial data which had been run out to time 10. In the case of the observational data, we simulate with a viscosity $\nu=\nu_{\text{fine}} = 2.75\times10^{-3}$, for which the energy spectrum is resolved up to machine precision. For some simulations, detailed below, we also considered a larger viscosity; namely $\nu_{\text{coarse}} = 2.75\times10^{-2}$.
We used a variable time stepper based on a CFL criterion.  The minimum time step, minimizing across all runs, was $\Delta t\approx 1.09896\times10^{-6}$.
We choose $C_s = 0.17$  as a standard model parameter for the DA solution (cf. \cite{Pope_2000_bible} c. 13). This yields that $\bar{\nu} = (C_s \delta)^2 = (C_s/256)^2 \approx 4.41 \times 10^{-7}$. For initial data for the DA solution, we use the standard initial data of $v_0 \equiv 0$. For $I_h$, we use a Fourier truncation, for which $c_0 = 1$. We observe Fourier modes of wave number $\mathbf{k}$ with $|\mathbf{k}| \leq 9$ (so $h = 1/9$) or less (at the end of the inertial range), yielding 2968 wave modes observed ($\approx$0.49\% of all active wave modes, accounting for symmetry, the 2/3's dealiasing, and the mean-zero condition). We use $\mu = 30$ which was chosen for convenience with no attempt to optimize the choice. Larger choices of $\mu$ are possible but we do not explore this here.

Since we are using a different code base, as a validation check we begin by corroborating the results of \cite{Cao_Giorgini_Jolly_Pakzad_2022}, where we consider observational data from the Ladyzhenskaya model being fed into the Ladyzhenskaya model. To this end, in Figure \ref{figLESLES}, we see that both relative $L^2$ error and relative $H^1$ error reach and stabilize near machine precision, both converging exponentially fast. 
\begin{figure}
    \centering
    \includegraphics[width=0.5\textwidth]{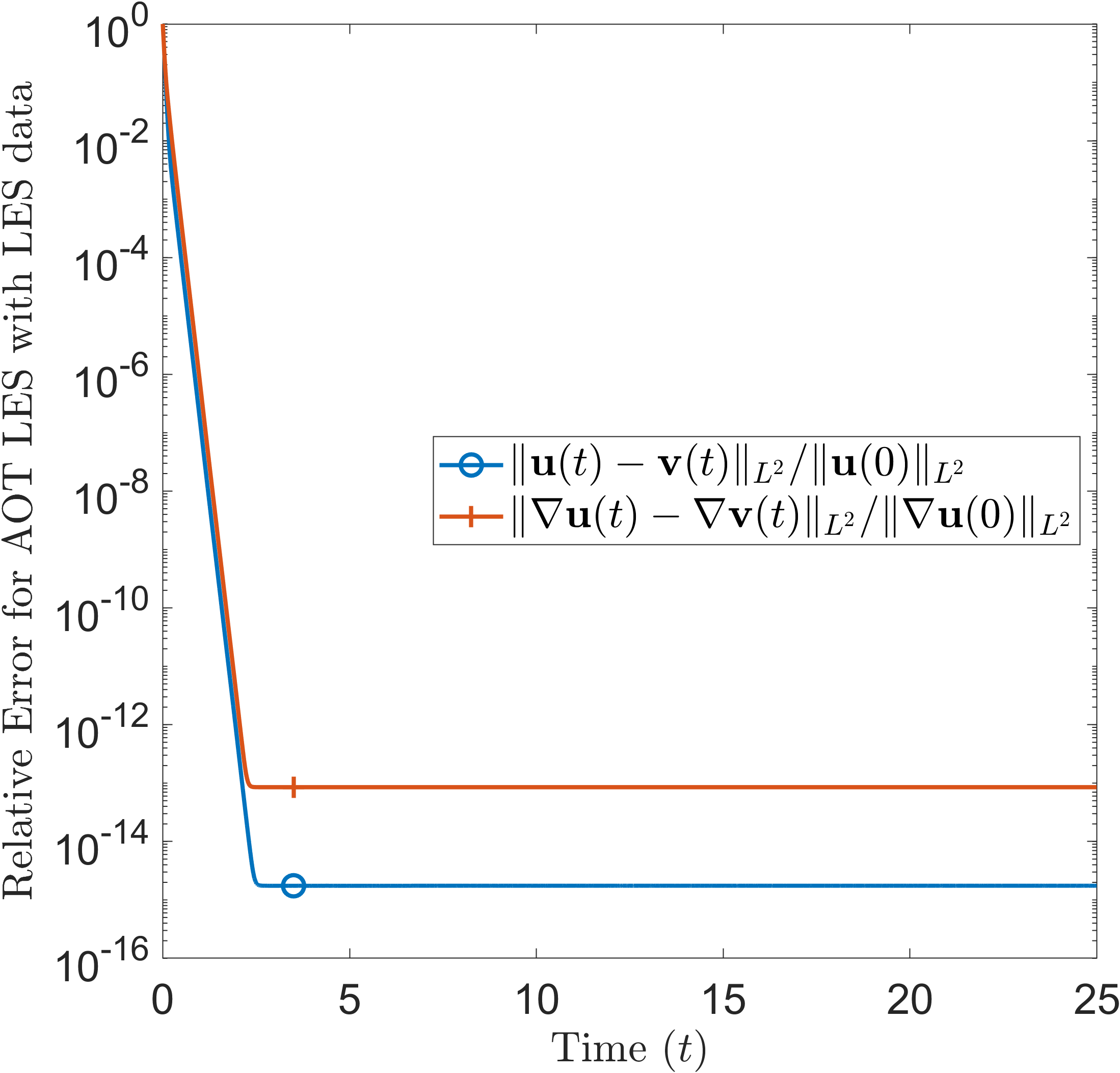}
    \caption{Shows relative $L^2$ and $H^1$ error for AOT applied to the Ladyzhenskaya model with observed data from the Ladyzhenskaya model.}
    \label{figLESLES}
\end{figure}

Next, in Figure \ref{fig:3Ddiffvisc} in blue we consider the AOT algorithm if one does not have enough resolution, and is instead forced to use a higher viscosity in the approximate solution. In this case, the different viscosity is $\nu = 2.75\times10^{-2}$, which we found to yield a resolved spectrum for a spatial resolution of $64^3$ grid points. In red we consider the model \ref{DALady} with the same viscosity as the natural viscosity. In yellow we consider the \ref{DALady} model with the different viscosity in the DA solution of $\nu = 2.75\times10^{-2}$. Lastly, in purple we show a reference of what happens if you simulate \ref{NSE} and the Ladyzhenskaya model from the same initial data with no data assimilation term and track the error. 

It is clear that all three methods of DA shown in Figure \ref{fig:3Ddiffvisc} present an advantage over simulating the Ladyzhenskaya model on its own. One important note to make is that the lines for AOT with different $\nu$ and AOT LES with different $\nu$ are not identical, having differences between each other of order $10^{-4}$ in both plots, yet they offer similar performance in terms of error. Regarding the simulation of \ref{DALady}, we see that a potential upper bound for $C_{\bfu,\Omega}\bar{\nu}$ is $\approx 10^{-4}$. Consistent with \ref{DAsynch}, we see exponential decay up to a model difference constant, and see the error level off as the high modes do not fully align. A visualization of this disparity is found in Figure \ref{3Dims}. Therein, we see contours of the magnitude of the vorticity of the observed solution (see Figure \ref{fig:3DObs}), the magnitude of the vorticity of the DA solution (see Figure \ref{fig:3DDA}), and the magnitude of the difference of the vorticities of the observed solution and the DA solution (see Figure \ref{fig3DDiff}). We see that relative the errors have magnitude up to $1.3 \times 10^{-3}$ and the errors form in areas of high vorticity. In Figure \ref{fig3DDiffZoom} we can see the jagged appearance of these regions of larger error, indicating a mismatch in  high wave modes.
\begin{figure}
    \centering
    \begin{subfigure}[t]{0.47\textwidth}
        \includegraphics[width=1\textwidth]{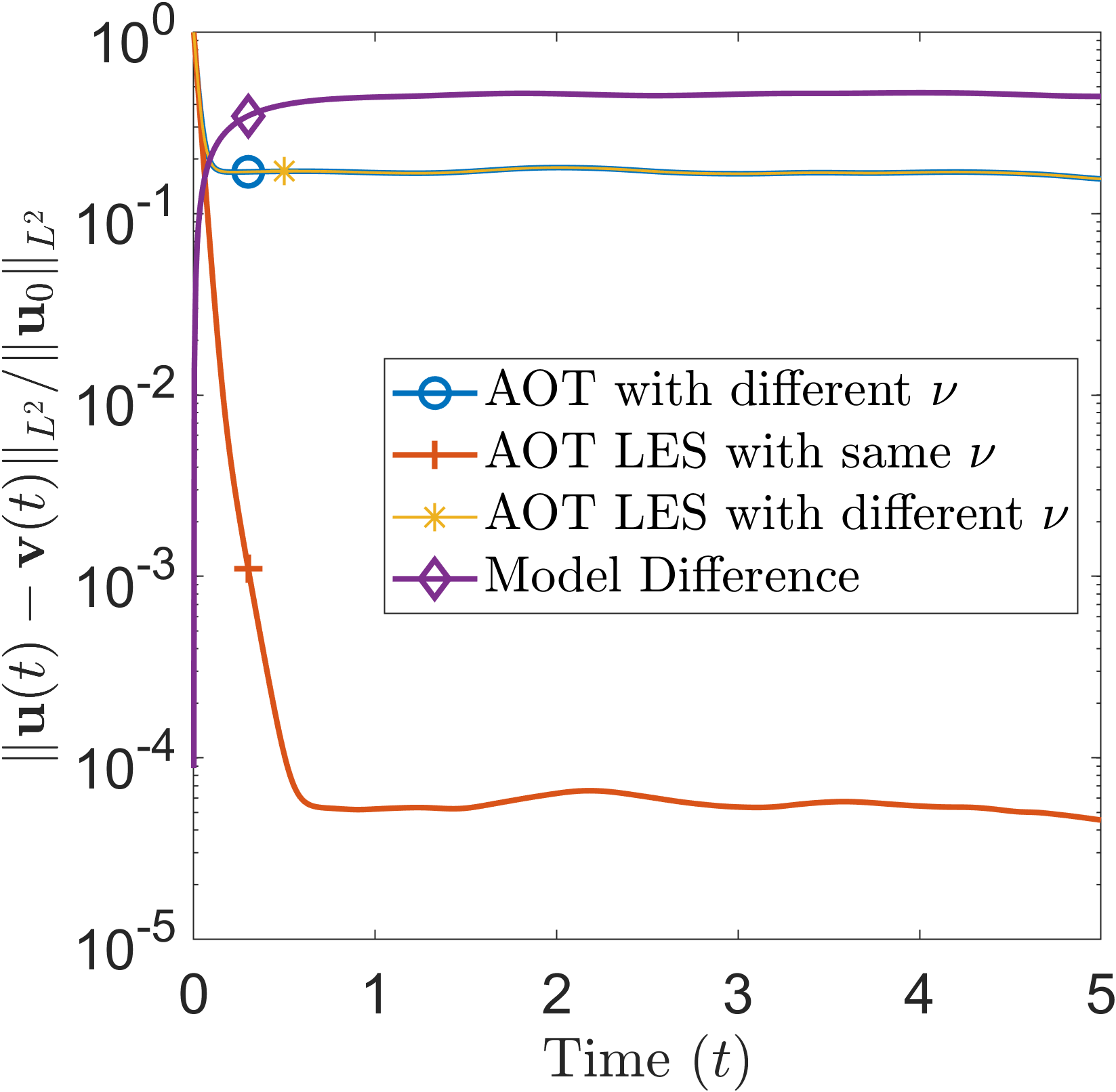}
        \caption[b]{\label{fig:3DL2Err}Relative $L^2$ error}
    \end{subfigure}
    \hspace{2mm}
    \begin{subfigure}[t]{0.47\textwidth}
        \includegraphics[width=1\textwidth]{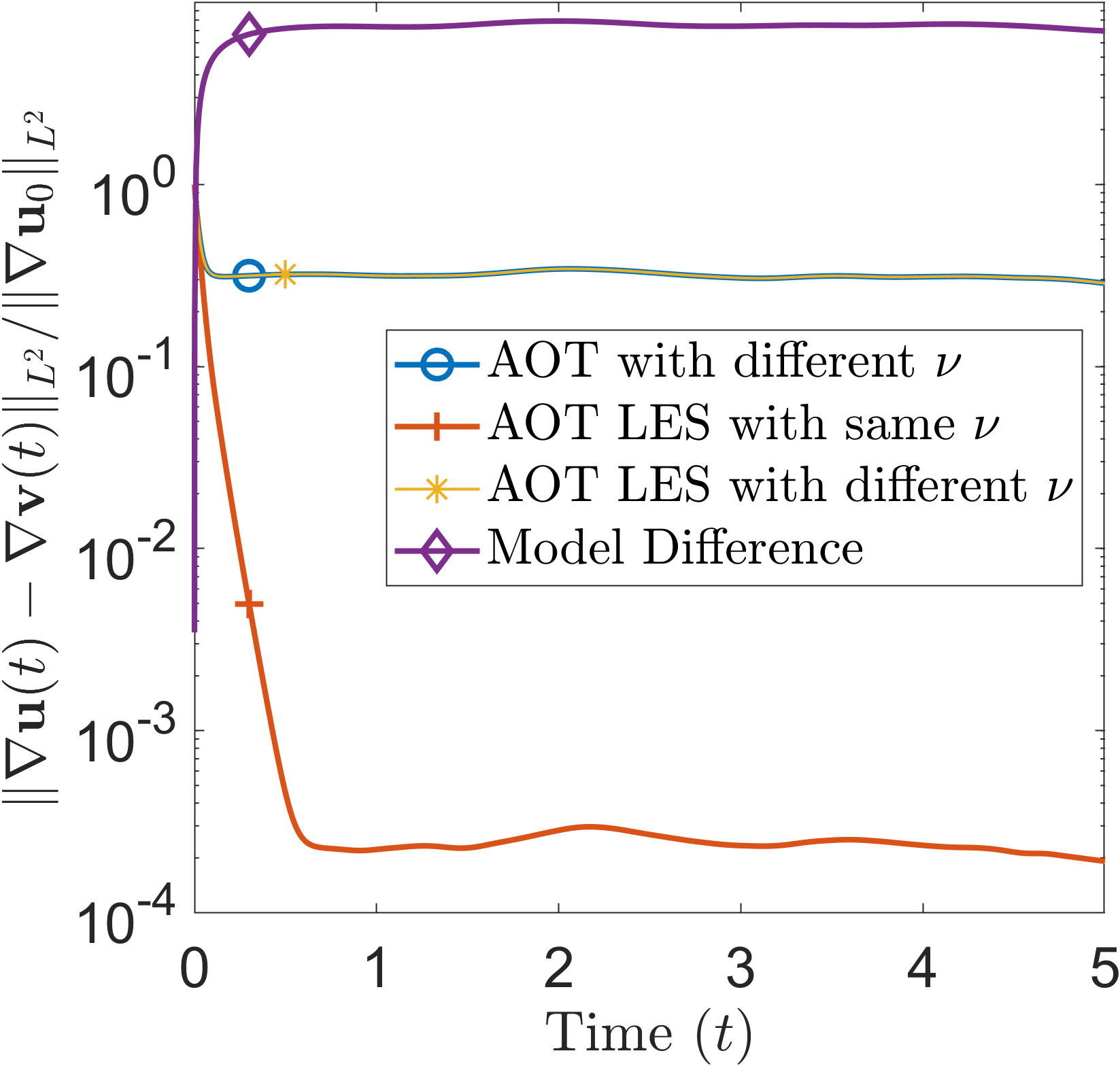}
        \caption[b]{\label{fig:3DH1Err}Relative $H^1$ error }
    \end{subfigure}
    \caption{For observed data coming with viscosity $\nu = 2.75\times10^{-3}$, we show relative error for AOT with a different viscosity $\nu = 2.75\times10^{-2}$, AOT LES with a different viscosity $2.75 \times 10^{-2}$ and the same viscosity $\nu = 2.75\times 10^{-3}$, and the Smagorinsky model with no DA.}
    \label{fig:3Ddiffvisc}
\end{figure}

\begin{figure}
    \centering
    \begin{subfigure}[t]{0.49\textwidth}
        \includegraphics[width=1\textwidth]{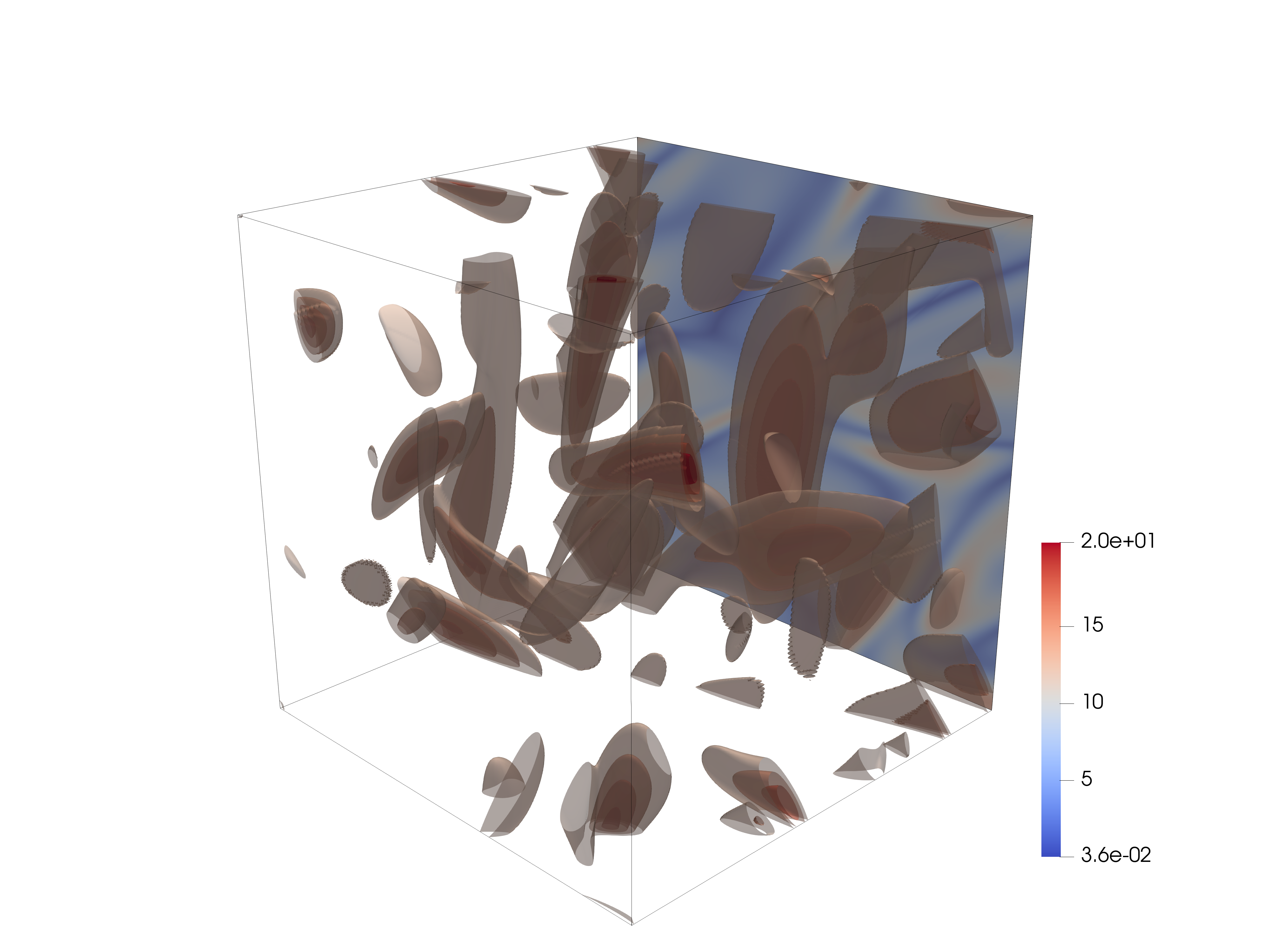}
        \caption[b]{\label{fig:3DObs}Observed Solution}
    \end{subfigure}
    \begin{subfigure}[t]{0.49\textwidth}
        \includegraphics[width=1\textwidth]{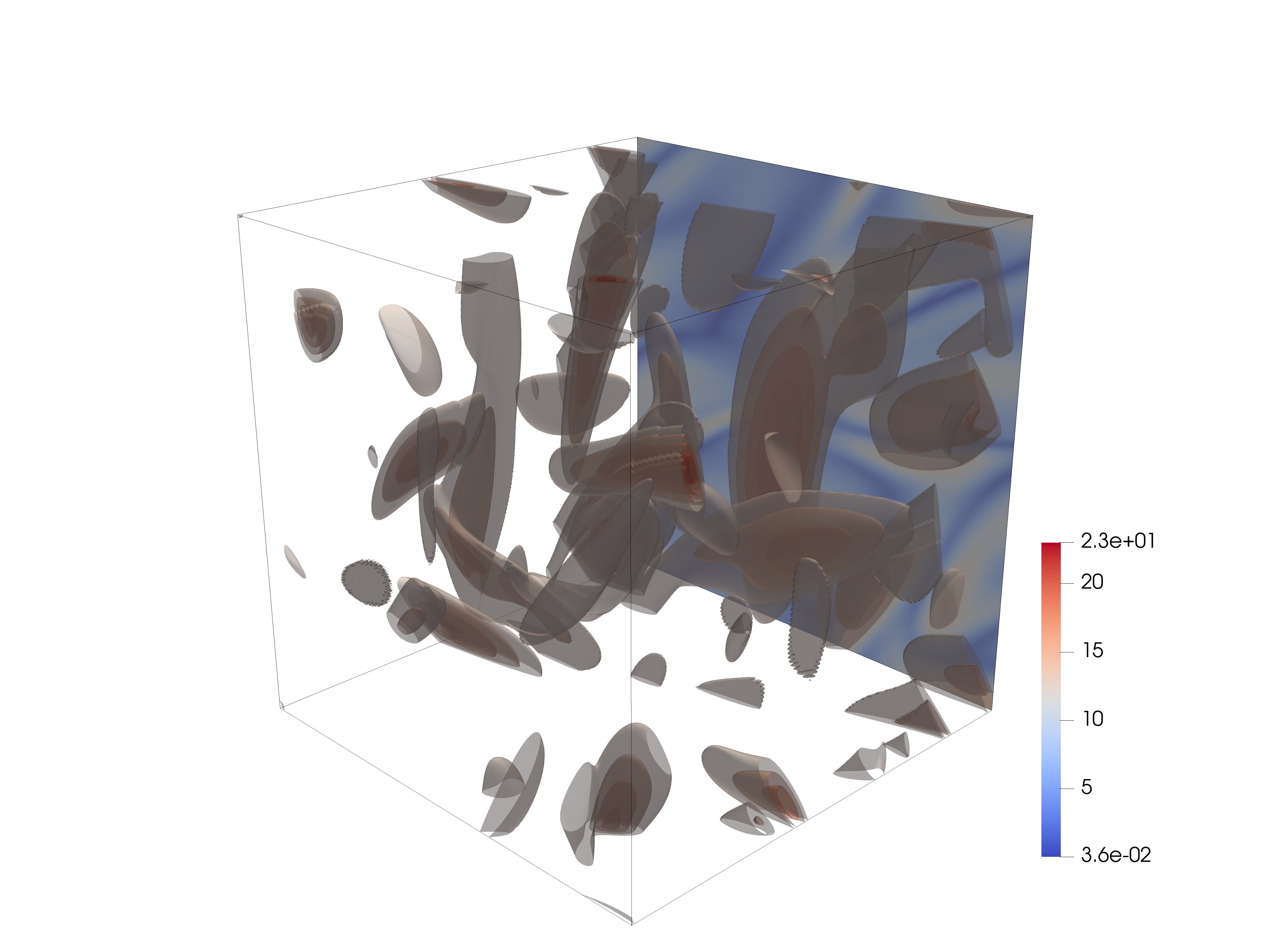}
        \caption[b]{\label{fig:3DDA}DA Solution}
    \end{subfigure}
    \begin{subfigure}[t]{0.49\textwidth}
        \includegraphics[width=1\textwidth]{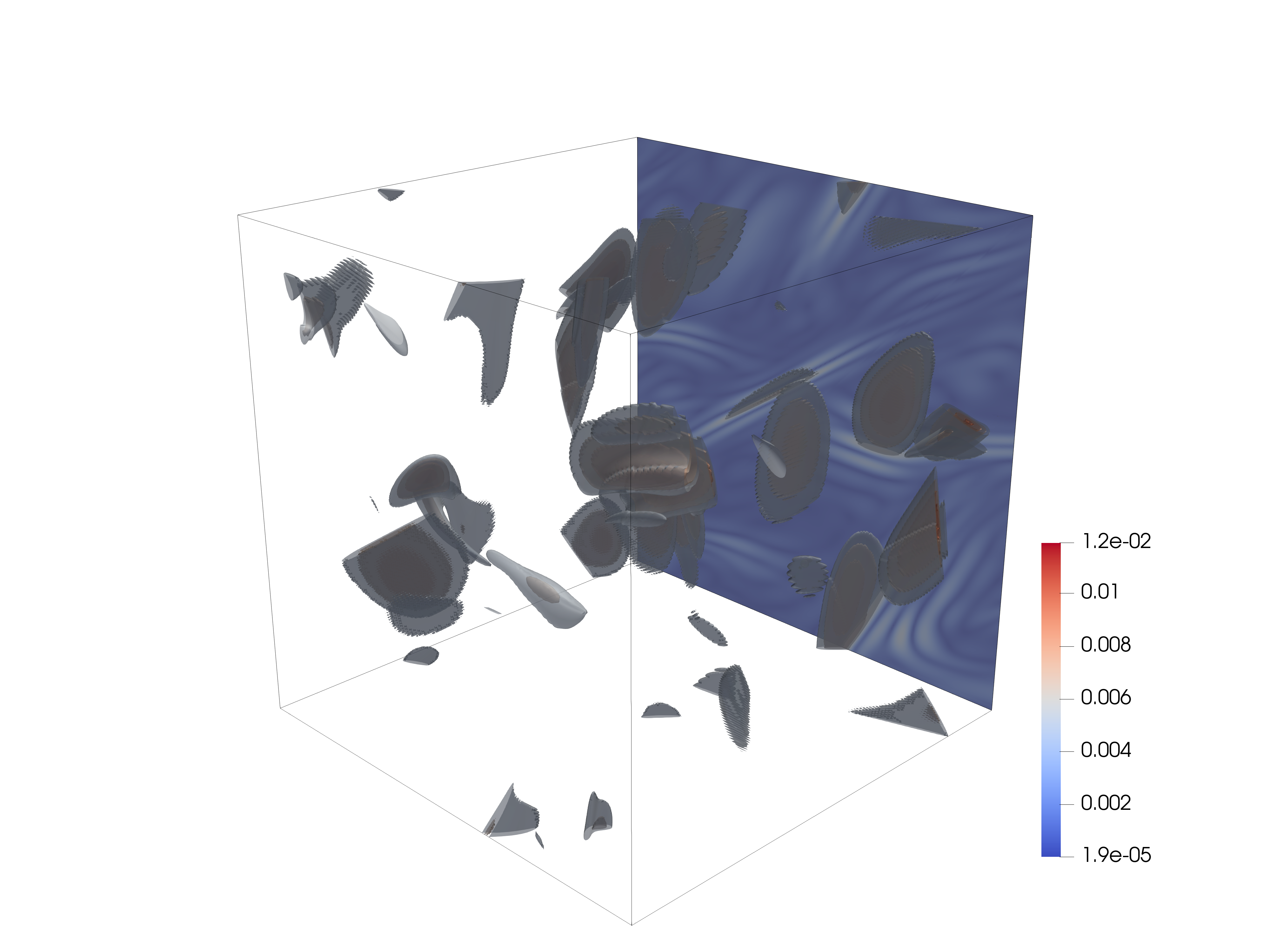}
        \caption[b]{\label{fig3DDiff}Difference}
    \end{subfigure}
            \hfill
    \begin{subfigure}[t]{0.42\textwidth}
        \includegraphics[width=1\textwidth,clip,trim=15mm 0mm 50mm 0mm]{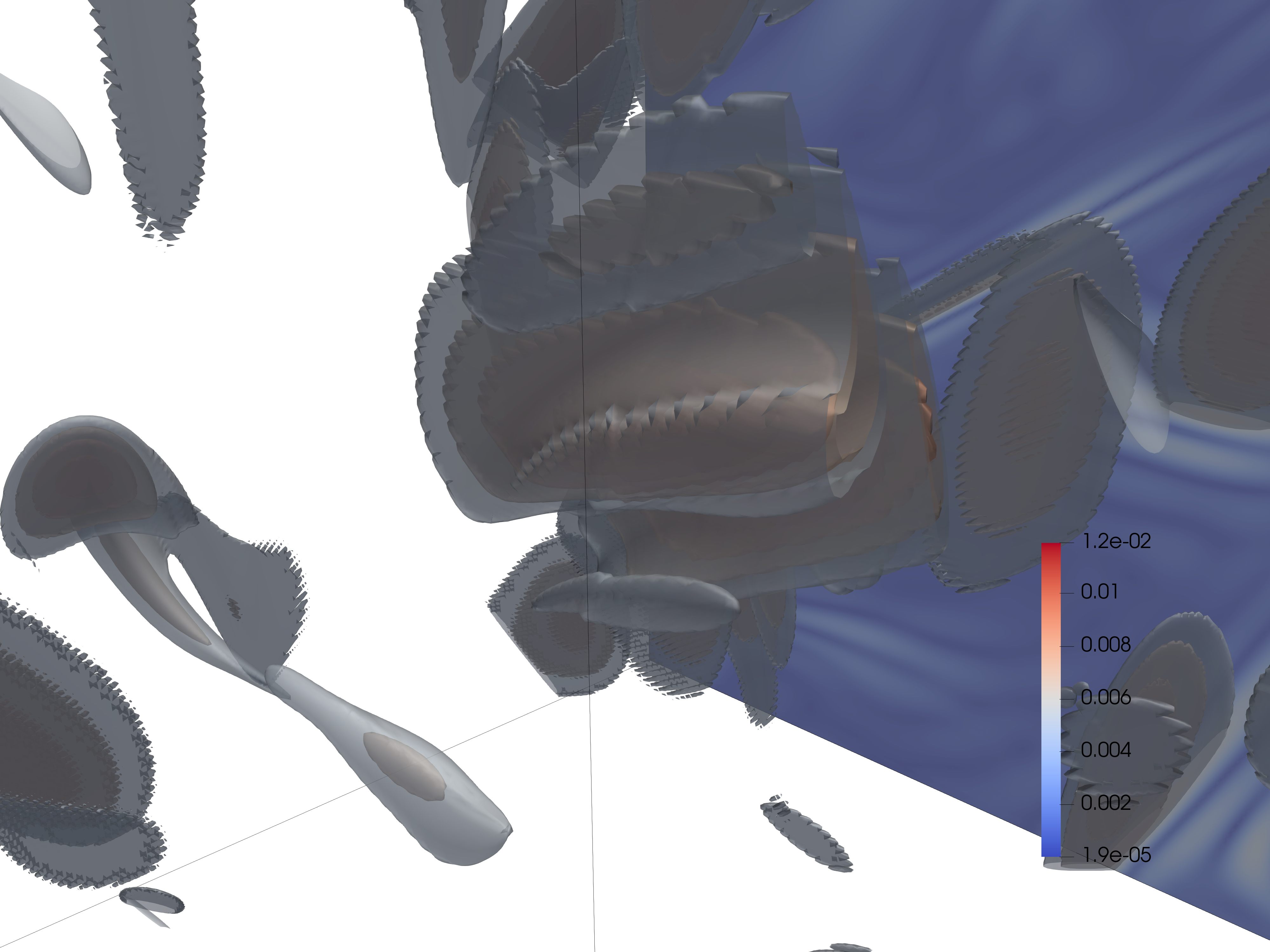}
        \caption[b]{\label{fig3DDiffZoom}Difference Zoomed In}
    \end{subfigure}
    \caption{Shows the magnitude of the vorticity of (A), the observed data, (B) the AOT-LES solution, and (C) the difference of the observed data and the AOT-LES solution. }
    \label{3Dims}
    
\end{figure}

\pagebreak

\end{document}